\documentclass[10pt]{amsart}
\usepackage{amssymb,amsmath}
\usepackage{hyperref}
\usepackage{bm}
\usepackage[]{graphicx}
\usepackage{amssymb, epsfig}
\newtheorem{theorem}{Theorem}[section]
\newtheorem{remm}{Remark}[section]
\newtheorem{lemma}{Lemma}[section]

\newtheorem{remark}[remm]{Remark}

\newtheorem{proposition}[theorem]{Proposition}

\numberwithin{equation}{section}
\newcommand{\rset}{{\mathbb R}}
\newcommand{\nset}{{\mathbb N}}

\newcommand{\bfdot}{\bf\dot}
\newcommand{\ken}{\ \ }

\newcommand{\ssy}{\scriptscriptstyle}

\newcommand{\half}{\frac{1}{2}}
\newcommand{\dtau}{\Delta\tau}
\begin{document}
\title[]
{Crank-Nicolson finite element approximations\\
for a linear stochastic fourth order equation \\
with additive space-time white noise$^{*}$}
\thanks{%
$^{*}$Work supported by the Research Grant no.3570/THALES-AMOSICSS to the University
of Crete co-funded by the European Union (European Social Fund-ESF) and Greek National
Funds.}
\author[]
{Georgios E. Zouraris$^{\dag}$}
\thanks{%
$^{\dag}$Department of Mathematics and Applied Mathematics,
University of Crete, PO Box 2208,
GR--710 03 Heraklion, Crete, Greece.
(e-mail: georgios.zouraris@uoc.gr)}

%
%
\maketitle
%
%
%
\begin{abstract}
We consider a model initial- and Dirichlet boundary- value problem for
a fourth-order linear stochastic parabolic equation, in one space
dimension, forced by an additive space-time white noise. 
First, we approximate its solution by the solution of an auxiliary
fourth-order stochastic parabolic problem with additive, finite dimensional,
spectral-type stochastic load. 
Then, fully-discrete approximations of the solution to the approximate
problem are constructed by using, for the discretization in space, a
standard Galerkin finite element method based on $H^2$-piecewise
polynomials, and, for time-stepping, the Crank-Nicolson method.
Analyzing the convergence of the proposed discretization approach,
we derive strong error estimates which show that the order of strong
convergence of the Crank-Nicolson finite element method is equal to
that reported in \cite{KZ2010} for the Backward Euler finite element method.
\end{abstract}
%
%
%
\section{Introduction}\label{SECT1}
%
%
%
%
%
Let $T>0$, $D=(0,1)$, $(\Omega,{\mathcal F},P)$ be a complete
probability space, and consider a model initial- and Dirichlet
boundary- value problem for a fourth-order linear stochastic
parabolic equation formulated as follows:
seek a stochastic function $v:[0,T]\times{\overline D}\to\rset$
such that
\begin{equation}\label{PARAP}
\begin{gathered}
v_t +v_{xxxx}={\dot W}\quad\text{\rm in}
\,\,\,(0,T]\times D,\\
v(t,\cdot)\big|_{\ssy\partial D}
=v_{xx}(t,\cdot)\big|_{\ssy\partial D}=0\quad\forall\,t\in(0,T],\\
v(0,x)=w_0\quad\forall\,x\in D,\\
\end{gathered}
\end{equation}
a.s. in $\Omega$, where ${\dot W}$ denotes a space-time
white noise on $[0,T]\times D$ (see, e.g., \cite{Walsh86},
\cite{KXiong}) and $w_0:{\overline D}\rightarrow{\mathbb R}$ is a
deterministic initial condition.
The mild solution of the problem above (cf. \cite{Carolina1},
\cite{DapDeb1}) has the form
\begin{equation*}
v=w+u
\end{equation*}
where:
\par\textbullet\quad $w:[0,T]\times{\overline D}\rightarrow{\mathbb R}$ is the solution
to the deterministic problem:
\begin{equation}\label{Det_Parab}
\begin{gathered}
w_t +w_{xxxx} = 0
\quad\text{\rm in}
\,\,\,(0,T]\times D,\\
w(t,\cdot)\big|_{\ssy\partial D}
=w_{xx}(t,\cdot)\big|_{\ssy\partial D}=0
\quad\forall\,t\in(0,T],\\
w(0,x)=w_0(x)\quad\forall\,x\in D,\\
\end{gathered}
\end{equation}
which is written as
\begin{equation}\label{deter_green}
w(t,x)=\int_{\ssy D}G(t;x,y)\,w_0(y)\,dy
\quad \forall\,(t,x)\in(0,T]\times{\overline D}
\end{equation}
with 
\begin{equation}\label{GreenKernel}
G(t;x,y)=\sum_{k=1}^{\infty}e^{-\lambda_k^4 t}
\,\varepsilon_k(x)\,\varepsilon_k(y)
\quad \forall\,(t,x)\in(0,T]\times{\overline D},
\end{equation}
and $\lambda_k:=k\,\pi$ for $k\in\nset$, and
$\varepsilon_k(z):=\sqrt{2}\,\sin(\lambda_k\,z)$
for $z\in{\overline D}$ and $k\in\nset$,
\par\noindent\vskip0.2truecm
\par\noindent
and
\par\noindent\vskip0.2truecm
\par\textbullet\quad  $u:[0,T]\times{\overline D}\rightarrow{\mathbb R}$ is
a stochastic function (known also as `{\it stochastic convolution}')
given by
\begin{equation}\label{MildSol}
u(t,x)=\int_0^t\!\int_{\ssy D}G(t-s;x,y)\,dW(s,y),
\end{equation}
which is, also,  the mild solution to the problem \eqref{PARAP}
when the initial condition $w_0$ vanishes.
\par
Thus, we can approximate numerically the mild solution $v$
by approximating separately the functions $w$ and $u$. 
In the work at hand, we focus on the development of a 
numerical method to approximate the stochastic part $u$
of the mild solution $v$ to the problem \eqref{PARAP}. In
particular, we will formulate and analyze
a numerical method which combines a Crank-Nicolson
time-stepping with a finite element method for
space discretization.
%
%
%
%
%
%
\subsection{An approximate problem for $u$}\label{maimou_problem}
%
%
\par
To construct computable approximations of $u$ we formulate
an auxiliary approximate stochastic fourth-order parabolic problem with
a finite dimensional additive noise inspired by the approach of 
\cite{ANZ} for the stochastic heat equation with additive
space-time white noise (cf. \cite{BinLi}, \cite{KZ2010},
\cite{KZ2013a}, \cite{KZ2013b}). 
%
%
%
%
\par\noindent
\hskip 1.0truecm\vbox{\hsize 14.0truecm\noindent\strut
Let ${\sf M}_{\star}\in\nset$ and
$S_{\ssy {\sf M}_{\star}}:=\mathop{\rm span}(\varepsilon_k)_{k=0}^{\ssy {\sf M}_{\star}}$.
Also, let ${\sf N}_{\star}\in\nset$, $\Delta{t}:=\frac{T}{{\sf N}_{\star}}$,
$t_n:=n\,\Delta{t}$ for $n=0,\dots,{\sf N}_{\star}$ be the
nodes of a uniform partition of the interval $[0,T]$ and
$T_n:=(t_{n-1},t_n)$ for $n=1,\dots,{\sf N}_{\star}$.
Then, we consider the fourth-order linear stochastic parabolic
problem:
\begin{equation}\label{AC2}
\begin{gathered}
{\widehat u}_t+{\widehat u}_{xxxx}={\widehat W}
\quad\text{\rm in}\,\,\, (0,T]\times D,\\
{\widehat u}(t,\cdot)\big|_{\ssy\partial D}=
{\widehat u}_{xx}(t,\cdot)\big|_{\ssy\partial D}=0
\quad\forall\,t\in(0,T],\\
{\widehat u}(0,x)=0\quad\forall\,x\in D,\\
\end{gathered}
\end{equation}
\par\noindent
a.e. in $\Omega$, where:
\begin{equation}\label{WNEQ1}
{\widehat W}(\cdot,x)\left|_{\ssy T_n}\right. :=
\tfrac{1}{\Delta{t}}\,\sum_{i=1}^{\ssy{\sf M}_{\star}}{\sf R}^n_i
\,\varepsilon_i(x)\quad\forall \,x\in D
\end{equation}
and
\begin{equation*}\label{WNEQ2}
{\sf R}^n_i:=\int_{\ssy T_n}\!\int_{\ssy D}\varepsilon_i(x)\;dW(t,x), \quad
i=1,\dots,{\sf M}_{\star},
\end{equation*}
for $n=1,\dots,{\sf N}_{\star}$.
\strut}
\par\noindent
The solution of the problem \eqref{AC2}, according to the standard
theory for parabolic problems (see, e.g, \cite{LMag}), has the
integral representation
\begin{equation}\label{HatUform}
\widehat{u}(t,x)= \int_0^t\!\!\!\int_{\ssy D} G(t-s;x,y)
\,{\widehat W}(s,y)\,dsdy
\quad\forall\,(t,x)\in[0,T]\times{\overline D}.
\end{equation}
\begin{remark}
Let $B^i(t):=\int_0^t\int_{\ssy D}\varepsilon_i(x)\;dW(s,x)$
for $t\ge0$ and $i\in{\mathbb N}$. According to \cite{Walsh86},
$(B^i)_{i\in{\mathbb N}}$ is a family of independent Brownian motions.
Thus, the random variables $\left(\left({\sf R}^n_i\right)_{n=1}^{\ssy {\sf N}_{\star}}\right)_{i\in{\mathbb N}}$
are independent and ${\sf R}^n_i\sim N(0,\Delta{t})$ for $i\in{\mathbb N}$ and $n=1,\dots,{\sf N}_{\star}$.
\end{remark}
\begin{remark}
The stochastic load ${\widehat W}$ in the right hand side of \eqref{AC2} corresponds
to a spectral-type representation of the space-time white noise.
We can, also, build up a numerical method for $u$ by using the approximate problem
proposed in Section~1.2 of \cite{KZ2010}, where 
the stochastic load ${\widehat W}$ is piecewise constant with respect to the time
variable but it is a discontinuous piecewise linear function with respect to the space variable. 
\end{remark}
%
%
\subsection{Crank-Nicolson fully discrete approximations}\label{CN_Main}
Let $M\in\nset$, $\dtau:=\tfrac{T}{M}$, $(\tau_m)_{m=0}^{\ssy M}$
be the nodes of a uniform partition of $[0,T]$ with width $\dtau$, i.e. 
$\tau_m:=m\,\dtau$ for $m=0,\dots,M$, and corresponding intervals
$\Delta_m:=(\tau_{m-1},\tau_m)$ for $m=1,\dots,M$.
For $p=2$ or $3$, let ${\sf S}_h^p\subset H^2(D)\cap H_0^1(D)$ be
a finite element space consisting of functions which are piecewise
polynomials of degree at most $p$ over a partition of $D$ in
intervals with maximum mesh-length $h$.
\par
The Crank-Nicolson finite element method to approximate the solution
${\widehat u}$ to the problem \eqref{AC2} is as follows:
\par
{\tt Step CN1}: Set
\begin{equation}\label{FullDE1}
U_h^0:=0.
\end{equation}
\par
{\tt Step CN2}: For $m=1,\dots,M$, find $U_h^m\in{\sf S}_h^p$ such that
\begin{equation}\label{FullDE2}
\left(U_h^m-U_h^{m-1},\chi\right)_{\ssy 0,D}
+\tfrac{\dtau}{2}\,{\mathcal B}\left(U_h^m+U_h^{m-1},\chi\right)
=\int_{\ssy\Delta_m} \left({\widehat W}(s,\cdot),\chi\right)_{\ssy 0,D}\,ds
\quad\forall\,\chi\in{\sf S}_h^p,
\end{equation}
where $(\cdot,\cdot)_{\ssy 0,D}$ is the usual $L^2(D)-$inner product
and ${\mathcal B}:H^2(D)\times H^2(D)\to{\mathbb R}$ is a bilinear form given
by
\begin{equation*}
{\mathcal B}(v_1,v_2):=(\partial_x^2v_1,\partial_x^2v_2)
\quad\forall\,v_1, v_2 \in H^2(D).
\end{equation*}
%
%
\subsection{Motivation, results and references}
%
%
\par
The recent research activity (see, e.g., \cite{YubinY03}, \cite{BinLi}, \cite{YubinY05}, \cite{Walsh05})
indicates that the Backward Euler finite element method, applied to the stochastic heat equation with
additive space-time white noise, has strong order of convergence equal to $\frac{1}{4}-\epsilon$
with respect to the time step $\Delta\tau$ and $\frac{1}{2}-\epsilon$ with respect to
maximum length $h$ of the subintervals of the partition used to construct the finite element spaces.
Both orders of convergence are optimal since they are consistent to the exponent of the
H{\"o}lder continuity property of the mild solution to the problem. 
The lack of smoothness for the mild solution is the reason that the strong order of convergence
of a numerical method that combines a high order time stepping with a finite
element space discretization, is expected to be equal to the strong order of convergence of the Backward
Euler finite element method.
However, the convergence analysis in \cite{Walsh05} provides a pessimistic strong error estimate 
for the Crank-Nicolson finite element method of the form
${\mathcal O}(\dtau\,h^{-\frac{3}{2}}+h^{\half})$, which introduces an uncertainty about the
convergence of the method when both $h$ and $\dtau$ freely tend to zero.
In addition, a bibliographical quest shows that the Crank-Nicolson method has been analyzed
in \cite{Erika02} and \cite{Erika03} under the assumption that the additive
space-time noise is smooth in space, while it is not among the time-discretization
methods analyzed in \cite{Printems} (see (3.10) in \cite{Printems}).
This unclear convergence behavior of the Crank-Nicolson
method, under the presence of an additive space-time white noise,
suggests a direction for further research.
%
%
%
\par
In the work at hand, we consider a different but similar problem,
the fourth order stochastic parabolic problem formulated in \eqref{PARAP},
motivated by the fact that its mild solution is one
of the components of the mild solution to the nonlinear stochastic Cahn-Hilliard
equation (see, e.g., \cite{DapDeb1}, \cite{Carolina1}).
%
%
%
We approximate the stochastic part $u$ of its mild solution $v$ by the
Crank-Nicolson finite element method formulated in Section~\ref{CN_Main},
for which we derive strong error estimates.
%
As a first step, we confirm that the solution ${\widehat u}$
to the approximate problem \eqref{AC2} is really an approximation of $u$
by estimating, in Theorem~\ref{BIG_Qewrhma1}
and in terms of ${\Delta t}$ and ${\sf M}_{\star}$, the difference $u-{\widehat u}$
in the $L^{\infty}_t(L^2_{\ssy P}(L^2_x))$ norm, arriving at the following
modeling error bound:
\begin{equation*}\label{Model_Error}
\max_{t\in[0,T]}\left[\,\int_{\ssy\Omega}\left(\int_{\ssy
D}|u(t,x)-{\widehat
u}(t,x)|^2\;dx\right)\,dP\,\right]^{\frac{1}{2}}\leq\,C\,\left(\,\delta^{-\half}
\,{\sf M}_{\star}^{-\frac{3}{2}+\delta}+{\Delta t}^{\frac{3}{8}}\,\right)
\quad\forall\,\delta\in\left(0,\tfrac{3}{2}\right]. 
\end{equation*}
Then, for  the Crank-Nicolson finite element approximations of ${\widehat u}$,
we derive (see Theorem~\ref{FFQEWR})  a discrete in time
$L^{\infty}_t(L^2_{\ssy P}(L^2_x))$ error estimate of the form:
%
%
%
%
\begin{equation*}\label{FDEstim4}
\max_{0\leq{m}\leq{\ssy M}}\left[\int_{\ssy\Omega}\left( \int_{\ssy
D}\big|U_h^m(x)-{\widehat
u}(\tau_m,x)\big|^2\;dx\right)dP\right]^{\frac{1}{2}}\leq\,C\,
\left(\,\epsilon_1^{-\half}\,\dtau^{\frac{3}{8}-\epsilon_1}
+\epsilon_2^{-\half}\,h^{\frac{p}{2}-\epsilon_2}\,\right)
\quad\forall\,\epsilon_1\in\left(0,\tfrac{3}{8}\right],
\quad\forall\,\epsilon_2\in\left(0,\tfrac{p}{2}\right].
\end{equation*}
%
%
%
The error estimate above, follows by estimating separately the
{\sl time discretization error} in Theorem~\ref{TimeDiscreteErr1}
and the {\sl space discretization error} in Theorem~\ref{Tigrakis}.
The definition of the aforementioned type of errors is made possible
by using the Crank-Nicolson time-discrete approximations
of ${\widehat u}$ introduced in Section~\ref{CN_Main_2}.
In particular, the {\sl time discretization error} is the approximation error
of the Crank-Nicolson time-discrete approximations and the
{\sl space discretization error} is the error between the Crank-Nicolson
fully discrete approximations and the Crank-Nicolson time discrete approximations.
In both cases, we use the Duhamel principle for the representation of the error
along with a low regularity nodal error estimate in a discrete in time $L^2_t(L^2_x)$
norm for modified Crank-Nicolson time discrete approximations of $w$ when
the {\sl time discretization} error is estimated (see Section~\ref{section3A})
and for modified Crank-Nicolson fully discrete approximations of $w$ when
the {\sl space discretization} error is estimated (see Section~\ref{SECTION44a}).
Roughly speaking, the error analysis for the Crank-Nicolson method
differs to that of the Backward Euler method, at the following points:
\par
$\diamond$ the numerical method that one has to analyze for the deterministic problem is 
a modification of the numerical method applied to the stochastic one
%
%
\par
and
\par
$\diamond$ the derivation of a low regularity  $L^2_t(L^2_x)$ nodal error estimate 
for the numerical method approximating the solution to the deterministic problem
is not a natural outcome of the stability properties of the method.
\par
The main outcome of the present work is that the strong order of convergence
of the Crank-Nicolson finite element method is equal to
the strong order of convergence of the Backward Euler finite element method,
which is due to the low regularity of $u$
(see, e.g., \cite{KZ2010}, \cite{LMes2011}).
Adapting properly the convergence analysis developed, we can improve the Crank-Nicolson
error estimate in \cite{Walsh05}, showing that the strong order of convergence of the
Crank-Nicolson finite element method applied to the stochastic heat equation with additive space-time white
noise is equal to the stong order of convergence of the Backward Euler finite
element method obtained in \cite{Walsh05} and \cite{YubinY05}.
Analogous result can be obtained for the linear fourth order problem \eqref{PARAP}
with additive derivative of a space-time white noise (cf. \cite{KZ2013b}), and
the two or three space dimension case of the
linear fourth order problem \eqref{PARAP} (cf. \cite{KZ2013a}).
\par
We close the section by a brief overview of the paper.
Section~\ref{SECTION_TWO} sets notation, recalls some
known results o\-ften used in the paper and introduce a usefull
projection operator.
Section~\ref{SECTION2} is dedicated to the estimation of the modeling error
$u-{\widehat u}$.
Section~\ref{SECTION3} defines the Crank-Nicolson time-discrete
approximations of ${\widehat u}$ and analyzes its convergence 
via the convergence analysis of modified Crank-Nicolson time-discrete
approximations of $w$.
%
%
Finally, Section~\ref{SECTION44} contains the error analysis for the
Crank-Nicolson fully-discrete approximations of ${\widehat u}$.
\section{Preliminaries}\label{SECTION_TWO}
We denote by $L^2(D)$ the space of the Lebesgue measurable
functions which are square integrable on $D$ with respect to
Lebesgue's measure $dx$, provided with the standard norm
$\|g\|_{\ssy 0,D}:= \left(\int_{\ssy D}|g(x)|^2\,dx\right)^{\half}$
for $g\in L^2(D)$. The standard inner product in $L^2(D)$ that
produces the norm $\|\cdot\|_{\ssy 0,D}$ is written as
$(\cdot,\cdot)_{\ssy 0,D}$, i.e., $(g_1,g_2)_{\ssy 0,D}
:=\int_{\ssy D}g_1(x)\,g_2(x)\,dx$ for $g_1$, $g_2\in L^2(D)$.
Let ${\mathbb N}_0$ be the set of the non negative integers.
Then, for $s\in\nset_0$, $H^s(D)$ will be the Sobolev space of functions
having generalized derivatives up to order $s$ in the space
$L^2(D)$, and by $\|\cdot\|_{\ssy s,D}$ its usual norm, i.e.
$\|g\|_{\ssy s,D}:=\left(\sum_{\ell=0}^s
\|\partial_x^{\ell}g\|_{\ssy 0,D}^2\right)^{\frac{1}{2}}$
for $g\in H^s(D)$. Also, by $H_0^1(D)$ we denote the subspace of $H^1(D)$
consisting of functions which vanish at the endpoints of $D$ in
the sense of trace.
%
%
%
\par
The sequence of pairs
$\left\{\left(\lambda_k^2,\varepsilon_k\right)\right\}_{k=1}^{\infty}$ is a
solution to the eigenvalue/eigenfunction problem: find nonzero
$\varphi\in H^2(D)\cap H_0^1(D)$ and $\sigma\in\rset$ such that
$-\varphi''=\sigma\,\varphi$ in $D$.
Since $(\varepsilon_k)_{k=1}^{\infty}$ is a complete
$(\cdot,\cdot)_{\ssy D}-$orthonormal system in $L^2(D)$, we define,
for $s\in\rset$, a subspace ${\mathcal V}^s(D)$ of $L^2(D)$ by
\begin{equation*}
{\mathcal V}^s(D):=\left\{g\in L^2(D):\quad\sum_{k=1}^{\infty}
\lambda_{k}^{2s}
\,(g,\varepsilon_k)^2_{\ssy 0,D}<\infty\,\right\}
\end{equation*}
provided with the norm
$\|g\|_{\ssy{\mathcal V}^s}:=\left(\,\sum_{k=1}^{\infty}
\lambda_{k}^{2s}\,(g,\varepsilon_k)^2_{\ssy
0,D}\,\right)^{\frac{1}{2}}$ for $g\in{\mathcal V}^s(D)$.
For $s\ge 0$, the pair $({\mathcal V}^s(D),\|\cdot\|_{\ssy{\mathcal V}^s})$ is a complete
subspace of $L^2(D)$ and we set
$({\bfdot H}^s(D),\|\cdot\|_{\ssy{\bfdot H}^s})
:=({\mathcal V}^s(D),\|\cdot\|_{\ssy{\mathcal V}^s})$.
For $s<0$, we define $({\bfdot H}^s(D),\|\cdot\|_{\ssy{\bfdot H}^s})$  as the completion of
$({\mathcal V}^s(D),\|\cdot\|_{\ssy{\mathcal V}^s})$, or, equivalently, as the dual of
 $({\bfdot H}^{-s}(D),\|\cdot\|_{\ssy{\bfdot H}^{-s}})$.
\par
Let $m\in\nset_0$. It is well-known (see \cite{Thomee}) that
\begin{equation}\label{dot_charact}
{\bfdot H}^m(D)=\big\{\,g\in H^m(D):
\quad\partial^{2i}_xg\left|_{\ssy\partial D}\right.=0
\quad\text{\rm if}\ken 0\leq{i}<\tfrac{m}{2}\,\big\}
\end{equation}
and there exist constants $C_{m,{\ssy A}}$ and $C_{m,{\ssy B}}$
such that
\begin{equation}\label{H_equiv}
C_{m,{\ssy A}}\,\|g\|_{\ssy m,D} \leq\|g\|_{\ssy{\bfdot H}^m}
\leq\,C_{m,{\ssy B}}\,\|g\|_{\ssy m,D}\quad \forall\,g\in{\bfdot
H}^m(D).
\end{equation}
Also, we define on $L^2(D)$ the negative norm $\|\cdot\|_{\ssy -m,
D}$ by
\begin{equation*}
\|g\|_{\ssy -m, D}:=\sup\Big\{ \tfrac{(g,\varphi)_{\ssy 0,D}}
{\|\varphi\|_{\ssy m,D}}:\quad \varphi\in{\bfdot H}^m(D)
\ken\text{\rm and}\ken\varphi\not=0\Big\} \quad\forall\,g\in
L^2(D),
\end{equation*}
for which, using \eqref{H_equiv}, it is easy to conclude that
there exists a constant $C_{-m}>0$ such that
\begin{equation}\label{minus_equiv}
\|g\|_{\ssy -m,D}\leq\,C_{-m}\,\|g\|_{{\bfdot H}^{-m}}
\quad\forall\,g\in L^2(D).
\end{equation}
\par
Let ${\mathbb L}_2=(L^2(D),(\cdot,\cdot)_{\ssy 0,D})$ and
${\mathcal L}({\mathbb L}_2)$ be the space of linear, bounded
operators from ${\mathbb L}_2$ to ${\mathbb L}_2$. We say that, an
operator $\Gamma\in {\mathcal L}({\mathbb L}_2)$ is 
Hilbert-Schmidt, when $\|\Gamma\|_{\ssy\rm
HS}:=\left(\sum_{k=1}^{\infty} \|\Gamma\varepsilon_k\|^2_{\ssy
0,D}\right)^{\half}<+\infty$, where $\|\Gamma\|_{\ssy\rm HS}$ is
the so called Hilbert-Schmidt norm of $\Gamma$.
%
%
We note that the quantity $\|\Gamma\|_{\ssy\rm HS}$ does not
change when we replace $(\varepsilon_k)_{k=1}^{\infty}$ by
another complete orthonormal system of ${\mathbb L}_2$.
It is well known (see, e.g., \cite{DunSch}) that an operator
$\Gamma\in{\mathcal L}({\mathbb L}_2)$ is Hilbert-Schmidt iff
there exists a measurable function $\gamma_{\star}:D\times D\rightarrow{\mathbb
R}$ such that $\Gamma[v](\cdot)=\int_{\ssy D}\gamma_{\star}(\cdot,y)\,v(y)\,dy$
for $v\in L^2(D)$, and then, it holds that
%
%
\begin{equation}\label{HSxar}
\|\Gamma\|_{\ssy\rm HS} =\left(\int_{\ssy D}\!\int_{\ssy
D}\gamma_{\star}^2(x,y)\,dxdy\right)^{\half}.
\end{equation}
Let ${\mathcal L}_{\ssy\rm HS}({\mathbb L}_2)$ be the set of
Hilbert Schmidt operators of ${\mathcal L}({\mathbb L}^2)$ and
$\Phi:[0,T]\rightarrow {\mathcal L}_{\ssy\rm HS}({\mathbb L}_2)$.
Also, for a random variable $X$, let ${\mathbb E}[X]$ be its
expected value, i.e., ${\mathbb E}[X]:=\int_{\ssy\Omega}X\,dP$.
Then, the It{\^o} isometry property for stochastic integrals,
which we will use often in the paper, reads
\begin{equation}\label{Ito_Isom}
{\mathbb E}\left[\Big\|\int_0^{\ssy T}\Phi\,dW\Big\|_{\ssy 0,D}^2\right]
=\int_0^{\ssy T}\|\Phi(t)\|_{\ssy\rm HS}^2\,dt.
\end{equation}
\par
We recall that: if $c_{\star}>0$, then
\begin{equation}\label{SR_BOUND}
\sum_{k=1}^{\infty}\lambda_k^{-(1+c_{\star}\delta)}
\leq\,\left(\tfrac{1+2c_{\star}}{c_{\star}\pi}\right)\,\tfrac{1}{\delta}
\quad\forall\,\delta\in(0,2],
\end{equation}
and if $({\mathcal H},(\cdot,\cdot)_{\ssy{\mathcal H}})$ is a
real inner product space, then
\begin{equation}\label{innerproduct}
(g-v,g)_{\ssy{\mathcal H}}
\ge\,\tfrac{1}{2}\,\left[\,(g,g)_{\ssy{\mathcal H}}
-(v,v)_{\ssy{\mathcal H}}\,\right]\quad\forall\,g,v\in{\mathcal H}.
\end{equation}
\par
For a nonempty set $A\subset[0,T]$, we will denote by ${\mathfrak X}_{\ssy A}:[0,T]\rightarrow\{0,1\}$
the indicator function of $A$. Also, for any $L\in{\mathbb N}$ and 
functions $(v^{\ell})_{\ell=0}^{\ssy L}\subset L^2(D)$ we define
$v^{\ell-\half}:=\tfrac{1}{2}(v^{\ell}+v^{\ell-1})$
for $\ell=1,\dots,L$. Finally, for $\alpha\in[0,1]$ and for $n=0,\dots,M-1$, 
we define $\tau_{n+\alpha}:=\tau_n+\alpha\,\dtau$. 
%
\subsection{A projection operator}
Let ${\mathfrak O}:=(0,T)\times D$, ${\mathfrak L}$ be a finite dimensional
subspace of $L^2({\mathfrak O})$ defined by  
\begin{equation*}
{\mathfrak L}:=\left\{\psi\in L^2({\mathfrak O}):\quad \exists (a^n)_{n=1}^{\ssy{\sf N}_{\star}}
\subset{\mathbb R}^{\ssy{{\sf M}_{\star}}}
\,\,\,\text{\rm s.t.}\,\,\,\psi(t,x)=\sum_{i=1}^{\ssy{\sf M}_{\star}}a^n_i\varepsilon_i(x)
\quad\forall\,(t,x)\in T_n\times D,\quad\, n=1,\dots,{\sf N}_{\star}\right\}
\end{equation*}
%
%
and $\Pi:L^2({\mathfrak O})\rightarrow{\mathfrak L}$ be the
$L^2({\mathfrak O})-$projection operator onto ${\mathfrak L}$ 
which is defined by requiring
\begin{equation*}
\int_0^{\ssy T}\!\int_{\ssy D}\Pi(g;t,x)\,\varphi(t,x)\;dtdx
=\int_0^{\ssy T}\!\int_{\ssy D}g(t,x)\,\varphi(t,x)\;dtdx
\quad\forall\,\varphi\in{\mathcal H},
\quad\forall\,g\in L^2({\mathfrak O}).
\end{equation*}
Then, we have
\begin{equation}\label{KatiL2}
\int_0^{\ssy T}\!\int_{\ssy D}(\Pi(g;t,x))^2\;dtdx\leq\int_0^{\ssy T}\!\int_{\ssy D}(g(t,x))^2\;dtdx
\quad\forall\,g\in L^2({\mathfrak O})
\end{equation}
and, after using a typical set of basis function for ${\mathfrak L}$, we, easily, conclude that
\begin{equation}\label{Defin_L2}
\Pi(g;t,x)=\tfrac{1}{\Delta t}\, \sum_{i=1}^{\ssy {\sf M}_{\star}}
\left(\int_{\ssy T_n}(g(s,\cdot),\varepsilon_i)_{\ssy 0,D}\;ds\right)\,\varepsilon_i(x)
\quad\forall(t,x)\in T_n\times D,
\,\,\, n=1,\dots,{\sf N}_{\star},
\,\,\,\forall\,g\in L^2({\mathfrak O}).
\end{equation}
\par
In the lemma below, we show a representation of the stochastic integral
of the projection $\Pi$ of a deterministic function $g\in L^2({\mathfrak O})$ as an
$L^2({\mathfrak O})-$inner product of $g$ with the random function
${\widehat W}$ defined in Section~\ref{maimou_problem}.
%
%
%
\begin{lemma}\label{Lhmma1}
Let ${\widehat W}$ be the random function defined in \eqref{WNEQ1}.
Then, it holds that
\begin{equation}\label{WNEQ2}
\int_0^{\ssy T}\!\!\!\int_{\ssy D}\,\Pi(g;s,y)\,dW(s,y)
=\int_0^{\ssy T}\!\!\!\int_{\ssy D}\,{\widehat W}(\tau,x)\,g(\tau,x)\;{d\tau}dx
\quad\forall\,g\in L^2({\mathfrak O}).
\end{equation}
\end{lemma}
%
%
%
%
%
\begin{proof}
Using \eqref{Defin_L2} and \eqref{WNEQ1}, we have
\begin{equation*}
\begin{split}
\int_0^{\ssy T}\!\!\!\int_{\ssy D}\Pi(g;s,y)\,dW(s,y)=&\,
\tfrac{1}{\Delta{t}}\,\sum_{n=1}^{\ssy{\sf N}_{\star}}\int_{\ssy T_n}\!\int_{\ssy D}\left[
\,\sum_{i=1}^{\ssy {\sf M}_{\star}}
\left(\int_{\ssy T_n}(g(\tau,\cdot),\varepsilon_i)_{\ssy 0,D}\;d\tau\right)\,
\varepsilon_i(y)\,\right]\;dW(s,y)\\
=&\,
\tfrac{1}{\Delta{t}}\,\sum_{n=1}^{\ssy{\sf N}_{\star}}
\sum_{i=1}^{\ssy {\sf M}_{\star}}
\left(\int_{\ssy T_n}(g(\tau,\cdot),\varepsilon_i)_{\ssy 0,D}\;d\tau\right)\,{\sf R}^n_i\\
=&\,
\tfrac{1}{\Delta{t}}\,\sum_{n=1}^{\ssy{\sf N}_{\star}}
\sum_{i=1}^{\ssy {\sf M}_{\star}}\left(
\int_{\ssy T_n}\!\int_{\ssy D}g(\tau,x)
\,{\sf R}^n_i\,\varepsilon_i(x)\;dxd\tau\right)\\
=&\,
\tfrac{1}{\Delta{t}}\,\sum_{n=1}^{\ssy{\sf N}_{\star}}\left(
\int_{\ssy T_n}\!\int_{\ssy D}
g(\tau,x)\,\left(\,\sum_{i=1}^{\ssy{\sf M}_{\star}}
\,{\sf R}^n_i\,\varepsilon_i(x)\,\right)\;dxd\tau\right)\\
=&\,\sum_{n=1}^{\ssy{\sf N}_{\star}}
\int_{\ssy T_n}\!\int_{\ssy D}
g(\tau,x)\,{\widehat W}(\tau,x)\;dxd\tau\\
=&\int_0^{\ssy T}\!\!\!\int_{\ssy D}g(\tau,x) \,{\widehat
W}(\tau,x)\,dtdx.
\end{split}
\end{equation*}
\end{proof}
%
%
%
\subsection{Linear elliptic and parabolic operators}\label{SECTION_III}
We denote by $T_{\ssy E}:L^2(D)\rightarrow{\bfdot H}^2(D)$ the solution
operator of the Dirichlet two-point boundary value problem: for given $f\in L^2(D)$ find $v_{\ssy E}\in {\bfdot
H}^2(D)$ such that
\begin{equation}\label{ElOp1}
v_{\ssy E}''=f\quad\text{\rm in}\ken D,
\end{equation}
i.e. $T_{\ssy E}f:=v_{\ssy E}$. Also, by
$T_{\ssy B}:L^2(D)\rightarrow{\bfdot H}^4(D)$ we denote the solution operator
of the Dirichlet biharmonic two-point boundary value problem: for given $f\in L^2(D)$
find $v_{\ssy B}\in {\bfdot H}^4(D)$ such that
\begin{equation}\label{ElOp2}
v_{\ssy B}''''=f\quad\text{\rm in}\ken D,
\end{equation}
i.e. $T_{\ssy B}f:=v_{\ssy B}$. 
Due to the type of boundary conditions of \eqref{ElOp2}, we
conclude that
\begin{equation}\label{tiger007}
T_{\ssy B}f= T_{\ssy E}^2f\quad\forall\,f\in L^2(D),
\end{equation}
which, easily, yields
\begin{equation}\label{TB-prop1}
(T_{\ssy B}v_1,v_2)_{\ssy 0,D} =(T_{\ssy E}v_1,T_{\ssy
E}v_2)_{\ssy 0,D}=(v_1,T_{\ssy B}v_2)_{\ssy 0,D}
\quad\forall\,v_1,v_2\in L^2(D).
\end{equation}
\par
It is well-known that the inverse elliptic operators
$T_{\ssy E}$ and $T_{\ssy B}$ satisfy the following
inequalities:
\begin{equation}\label{ElReg1}
\|T_{\ssy E}f\|_{\ssy m,D}\leq \,C_{\ssy E} \,\|f\|_{\ssy m-2, D}
\quad\forall\,f\in H^{\max\{0,m-2\}}(D), \ken\forall\,m\in{\mathbb
N}_0
\end{equation}
and
\begin{equation}\label{ElBihar2}
\|T_{\ssy B}f\|_{\ssy m,D}\leq \,C_{\ssy B}\,\|f\|_{\ssy m-4, D}
\quad\forall\,f\in H^{\max\{0,m-4\}}(D),
\quad\forall\,m\in{\mathbb N}_0,
\end{equation}
where the nonnegative constants $C_{\ssy E}$ and $C_{\ssy B}$ depend only on $D$. 
%
\par
Let $\left({\mathcal S}(t)w_0\right)_{t\in[0,T]}$ be the standard semigroup
notation for the solution $w$ of \eqref{Det_Parab}. 
For $\ell\in{\mathbb N}_0$, $\beta\ge0$, $r\ge0$ and
$q\in[0,r+4\ell]$ there exists a constant $C_{r,q,\ell}>0$
(see, e.g., Appendix A in \cite{KZ2008}, \cite{Thomee}, \cite{Pazy}) 
such that
\begin{equation}\label{Reggo0}
\big\|\partial_t^{\ell}{\mathcal S}(t)w_0\big\|_{\ssy {\bfdot
H}^r} \leq \,C_{r,q,\ell}\,\,\,t^{-\frac{r-q}{4}-\ell}
\,\,\,\|w_0\|_{\ssy {\bfdot
H}^q}\quad\forall\,t>0,\ken\forall\,w_0\in{\bfdot H}^q(D),
\end{equation}
and a constant $C_{\beta}>0$ such that
\begin{equation}\label{Reggo3}
\int_{t_a}^{t_b}(\tau-t_a)^{\beta}\,
\big\|\partial_t^{\ell}{\mathcal S}(\tau)w_0 \big\|_{\ssy {\bfdot
H}^r}^2\,d\tau \leq\,C_{\beta}\, \|w_0\|^2_{\ssy {\bfdot
H}^{r+4\ell-2\beta-2}} \quad\forall\,t_b>t_a\ge0,
\ken\forall\,w_0\in{\bfdot H}^{r+4\ell-2\beta-2}(D).
\end{equation}
%
\subsection{Discrete operators}\label{SECTION41}
%
%
Let $p=2$ or $3$, and ${\sf S}_h^p\subset H^2(D)\cap H_0^1(D)$
be a finite element space consisting of functions which are piecewise
polynomials of degree at most $p$ over a partition of $D$ in
intervals with maximum mesh-length $h$.
It is well-known (see, e.g., \cite{Cia}, \cite{BrScott}) that the following approximation
property holds: there exists a constant $C_{{\ssy{\sf FM}},p}>0$ such that
\begin{equation}\label{pin0}
\inf_{\chi\in{\sf S}_h^p}\|v-\chi\|_{\ssy 2,D}
\leq\,C_{{\ssy{\sf  FM}},p}\,h^{\ell-2}\,\|v\|_{\ssy \ell,D}
\quad\,\forall\,v\in H^{\ell}(D)\cap H_0^1(D), \quad \ell=3,\dots,p+1.
\end{equation}
Then, we define the discrete biharmonic operator
$B_h:{\sf S}_h^p\to{{\sf S}_h^p}$ by
$(B_h\varphi,\chi)_{\ssy 0,D} =(\partial_x^2\varphi,\partial_x^2\chi)_{\ssy 0,D}$
for $\varphi,\chi\in{\sf S}_h^p$,
the $L^2(D)-$projection operator $P_h:L^2(D)\to{\sf S}_h^p$ by
$(P_hf,\chi)_{\ssy 0,D}=(f,\chi)_{\ssy 0,D}$ for $\chi\in{\sf S}_h^p$
and $f\in L^2(D)$,
and the standard Galerkin finite element approximation
$v_{\ssy B,h}\in{\sf S}_h^p$ of the solution $v_{\ssy B}$
of \eqref{ElOp2} by requiring
\begin{equation}\label{fem2}
B_hv_{\ssy B,h}=P_hf.
\end{equation}
Letting  $T_{\ssy B,h}:L^2(D)\to{\sf S}^p_h$ be the solution operator of
the finite element method \eqref{fem2}, i.e.,
\begin{equation*}
T_{\ssy B,h}f:=v_{\ssy B,h}=B_h^{-1}P_hf
\quad\forall\,f\in L^2(D),
\end{equation*}
we can easily conclude that
\begin{equation}\label{adjo2}
(T_{\ssy B,h}f,g)_{\ssy 0,D}=\left(\partial_x^2(T_{\ssy B,h}f),
\partial_x^2(T_{\ssy B,h}g)\right)_{\ssy 0,D}
\quad\forall\,f,g\in L^2(D).
\end{equation}
Also, using the approximation property \eqref{pin0} of the finite element space
${\sf S}_h^p$, we can prove (see, e.g., Proposition 2.2 in \cite{KZ2010})
the following $L^2(D)-$error estimate for the finite element method \eqref{fem2}:
%
%
\begin{equation}\label{ARA1}
\|T_{\ssy B}f-T_{\ssy B,h}f\|_{\ssy 0,D}
\leq\,C\,h^p\,\|f||_{\ssy -1,D}\quad\forall\,f\in L^2(D).
\end{equation}
%
%
%
%
Observing that the Galerkin orthogonality property reads
\begin{equation*}
{\mathcal B}\left(T_{\ssy B}f-T_{\ssy B,h}f,\chi\right)_{\ssy 0,D}=0
\quad\forall\,\chi\in{\sf S}_h^p,\quad\forall\,f\in L^2(D),
\end{equation*}
after setting $\chi=T_{\ssy B,h}f$ and using the Cauchy-Schwarz inequality
along with \eqref{ElBihar2}, we get
\begin{equation}\label{TB_bound}
\begin{split}
\|\partial_x^2(T_{\ssy B,h}f)\|\leq&\,\|\partial_x^2(T_{\ssy B}f)\|_{\ssy 0,D}\\
\leq&\,C\,\|f\|_{\ssy -2,D}\quad\forall\,f\in L^2(D).\\
\end{split}
\end{equation}
%
%
%
%
%
%
\section{An estimate of the error $u-{\widehat u}$}\label{SECTION2}
In the theorem below, we derive an $L^{\infty}_t(L^2_{\ssy
P}(L^2_x))$ bound for the difference $u-{\widehat u}$ in terms of $\Delta{t}$
and ${\sf M}_{\star}$.
%
%
%
%
\begin{theorem}\label{BIG_Qewrhma1}
Let $u$ be the stochastic function defined by \eqref{MildSol} and
${\widehat u}$ be the solution of \eqref{AC2}.
Then, there exists a constant $C>0$, independent
of $\Delta{t}$ and ${\sf M}_{\star}$, such that
\begin{equation}\label{ModelError}
\max_{\ssy t\in[0,T]}\left(\,{\mathbb E}\left[\|u(t,\cdot)
-{\widehat u}(t,\cdot)\|_{\ssy 0,D}^2\right]\,\right)^{\half}
\leq\,C\,\left(\,\Delta{t}^\frac{3}{8}+\delta^{-\half}
\,{\sf M}_{\star}^{-\frac{3}{2}+\delta}\,\right)
\quad\forall\,\delta\in\left(0,\tfrac{3}{2}\right].
\end{equation}
\end{theorem}
%
%
%
%
%
%
\begin{proof}
%
%
%
%
%
Let ${\sf Z}(t):=\left({\mathbb E}\left[\|u(t,\cdot)-{\widehat u}(t,\cdot)\|_{\ssy 0,D}^2
\right]\right)^{\half}$ for $t\in[0,T]$. We will get \eqref{ModelError} working with the 
representations \eqref{MildSol} and \eqref{HatUform}. In the sequel,  we will use the symbol
$C$ to denote a generic constant that is independent of $\Delta{t}$ and ${\sf M}_{\star}$
and may changes value from one line to the other.
\par
Using \eqref{MildSol}, \eqref{HatUform}, \eqref{WNEQ2}
and \eqref{Defin_L2}, we conclude that
\begin{equation}\label{corv0}
u(t,x)-{\widehat u}(t,x)=\int_0^{\ssy T}\!\!\!\int_{\ssy D} \big[{\mathfrak
X}_{(0,t)}(s)\,G(t-s;x,y) -{\mathfrak G}(t,x;s,y)\big]\,dW(s,y)
\quad\forall\,(t,x)\in(0,T]\times D,
\end{equation}
where ${\mathfrak G}:(0,T]\times D\rightarrow L^2({\mathfrak O})$
given by
\begin{equation}\label{frak_G}
{\mathfrak G}(t,x;s,y)
:=\tfrac{1}{\Delta{t}}\,\sum_{i=1}^{\ssy {\sf M}_{\star}}
\left[\,\int_{\ssy T_n}{\mathfrak X}_{(0,t)}(s')\left(\int_{\ssy D}
\,G(t-s';x,y')\,\varepsilon_i(y')\;dy'\right)ds'\,\right]
\,\varepsilon_i(y)\quad\forall\,(s,y)\in T_n\times D
\end{equation}
for $(t,x)\in(0,T]\times D$ and $n=1,\dots,{\sf N}_{\star}$.
Thus, using \eqref{corv0} and \eqref{Ito_Isom}, we obtain
\begin{equation*}
{\sf Z}(t)=\left(\int_0^{\ssy T}\left(\int_{\ssy D}\int_{\ssy D}
 \big[{\mathfrak
X}_{(0,t)}(s)\,G(t-s;x,y) -{\mathfrak G}(t,x;s,y)\big]^2\;dxdy\right)ds\right)^{\frac{1}{2}}
\quad\forall\,t\in(0,T].
\end{equation*}
%
%
Now, we introduce the splitting
\begin{equation}\label{corv1}
{\sf Z}(t)\leq\,{\sf Z}_{\ssy A}(t)
+{\sf Z}_{\ssy B}(t)\quad\forall\,t\in(0,T],
\end{equation}
where
\begin{equation}\label{ZaZa}
{\sf Z}_{\ssy A}(t):=\left\{
\sum_{n=1}^{{\sf N}_{\star}}\int_{\ssy D}\!\int_{\ssy D}\!\int_{\ssy T_n}
\left[{\mathfrak X}_{(0,t)}(s)\,G(t-s;x,y)-\tfrac{1}{\Delta{t}}\,\int_{\ssy T_n}
{\mathfrak X}_{(0,t)}(s')\,G(t-s';x,y)ds'\right]^2\;dxdyds\right\}^{\frac{1}{2}}
\end{equation}
and
\begin{equation}\label{ZbZb}
{\sf Z}_{\ssy B}(t):=\left\{
\sum_{n=1}^{{\sf N}_{\star}}
\int_{\ssy D}\int_{\ssy D}\int_{\ssy T_n}\left[
\tfrac{1}{\Delta{t}}\int_{\ssy T_n}{\mathfrak X}_{(0,t)}(s')\,G(t-s';x,y)ds'
-{\mathfrak G}(t,x;s,y)\right]^2dxdyds
\right\}^{\frac{1}{2}}.
\end{equation}
Using \eqref{frak_G} and the $L^2(D)-$orthogonality of
$(\varepsilon_k)_{k=1}^{\infty}$  we obtain
\begin{equation}\label{X_Nov_2015}
{\mathfrak G}(t,x;s,y)
=\tfrac{1}{\Delta{t}}\,
\int_{\ssy T_n}{\mathfrak X}_{(0,t)}(s')\left[\,
\sum_{i=1}^{\ssy{\sf M}_{\star}}e^{-\lambda_i^4(t-s')}
\varepsilon_i(x)\,\varepsilon_i(y)\right]ds'\quad\forall\,(s,y)\in T_n\times D
\end{equation}
for $(t,x)\in(0,T]\times D$ and $n=1,\dots,{\sf N}_{\star}$.
Next, we combine \eqref{ZbZb} and \eqref{X_Nov_2015}
and use, again,
the $L^2(D)-$orthogonality of $(\varepsilon_k)_{k=1}^{\infty}$
to get
\begin{equation*}
\begin{split}
{\sf Z}_{\ssy B}(t)=&\,\left\{\tfrac{1}{\Delta{t}}\,
\sum_{n=1}^{{\sf N}_{\star}}
\int_{\ssy D}\int_{\ssy D}\left[
\int_{\ssy T_n}{\mathfrak X}_{(0,t)}(s')
\left(\,G(t-s';x,y)-\sum_{i=1}^{\ssy{\sf M}_{\star}}e^{-\lambda_i^4(t-s')}
\varepsilon_i(x)\,\varepsilon_i(y)\right)ds'
\right]^2dxdy
\right\}^{\frac{1}{2}}\\
=&\,\left\{\tfrac{1}{\Delta{t}}\,
\sum_{n=1}^{{\sf N}_{\star}}
\int_{\ssy D}\int_{\ssy D}\left[
\int_{\ssy T_n}{\mathfrak X}_{(0,t)}(s')
\left(\,\sum_{i={{\sf M}_{\star}}+1}^{\infty}e^{-\lambda_i^4(t-s')}
\varepsilon_i(x)\,\varepsilon_i(y)\right)ds'
\right]^2dxdy
\right\}^{\frac{1}{2}}\\
=&\,\left\{\tfrac{1}{\Delta{t}}\,
\sum_{n=1}^{{\sf N}_{\star}}
\int_{\ssy D}\int_{\ssy D}\left[
\,\sum_{i={{\sf M}_{\star}}+1}^{\infty}
\left( \int_{\ssy T_n}{\mathfrak X}_{(0,t)}(s')\,e^{-\lambda_i^4(t-s')}\;ds' \right)
\varepsilon_i(x)\,\varepsilon_i(y)
\right]^2dxdy
\right\}^{\frac{1}{2}}\\
=&\,\left\{\tfrac{1}{\Delta{t}}\,
\sum_{n=1}^{{\sf N}_{\star}}
\int_{\ssy D}\left[
\,\sum_{i={{\sf M}_{\star}}+1}^{\infty}
\left( \int_{\ssy T_n}{\mathfrak X}_{(0,t)}(s')\,e^{-\lambda_i^4(t-s')}\;ds' \right)^2
\varepsilon_i^2(x)\right]\;dx
\right\}^{\frac{1}{2}}\\
=&\,\left\{\tfrac{1}{\Delta{t}}\,
\sum_{n=1}^{{\sf N}_{\star}}
\,\sum_{i={{\sf M}_{\star}}+1}^{\infty}
\left( \int_{\ssy T_n}{\mathfrak X}_{(0,t)}(s')\,e^{-\lambda_i^4(t-s')}\;ds' \right)^2
\right\}^{\frac{1}{2}}\quad\forall\,t\in(0,T].
\end{split}
\end{equation*}
Then, using the Cauchy-Schwarz inequality and \eqref{SR_BOUND}, we obtain
\begin{equation}\label{ZbZb_bound}
\begin{split}
{\sf Z}_{\ssy B}(t)\leq&\,
\,\left\{
\,\sum_{i={{\sf M}_{\star}}+1}^{\infty}
\left(\int_{0}^te^{-2\,\lambda_i^4(t-s')}\;ds'\right)
\right\}^{\frac{1}{2}}\\
\leq&\,\left(\,
\sum_{i={\sf M}_{\star}}^{\infty}\tfrac{1}{2\,\pi^4\,i^4}\,\right)^{\half}\\
\leq&\,\tfrac{1}{\sqrt{2}\,\pi^2}\,\tfrac{1}{{\sf M}_{\star}^{\frac{3}{2}-\delta}}
\left(\,\sum_{i={\sf M}_{\star}}^{\infty}\tfrac{1}{i^{1+2\,\delta}}\,\right)^{\half}\\
\leq&C\,\delta^{-\half}\,{\sf M}_{\star}^{-\frac{3}{2}+\delta}
\quad\forall\,t\in(0,T],\quad\forall\,\delta\in\left(0,\tfrac{3}{2}\right].\\
\end{split}
\end{equation}
Finally, combining \eqref{ZaZa} along with the
$L^2(D)-$orthogonality of $(\varepsilon_k)_{k=1}^{\infty}$,
we conclude that
\begin{equation*}
\begin{split}
{\sf Z}_{\ssy A}(t)=&\,\left\{\,\tfrac{1}{(\Delta{t})^2}\,
\sum_{n=1}^{{\sf N}_{\star}}\int_{\ssy D}\!\int_{\ssy D}\!\int_{\ssy T_n}
\left[\int_{\ssy T_n}\left[{\mathfrak X}_{(0,t)}(s)\,G(t-s;x,y)-
{\mathfrak X}_{(0,t)}(s')\,G(t-s';x,y)\right]ds'\right]^2\;dxdyds\right\}^{\frac{1}{2}}\\
=&\,\left\{\,\sum_{i=1}^{\infty}\tfrac{1}{(\Delta{t})^2}\,
\sum_{n=1}^{{\sf N}_{\star}}\int_{\ssy T_n}
\left(\int_{\ssy T_n}\left[{\mathfrak X}_{(0,t)}(s)\,e^{-\lambda_i^4(t-s)}-
{\mathfrak X}_{(0,t)}(s')\,e^{-\lambda_i^4(t-s')}\right]ds'\right)^2
\;ds\right\}^{\frac{1}{2}}
\quad\forall\,t\in(0,T].\\
\end{split}
\end{equation*}
Then, we proceed as in the proof of Theorem~3.1 in \cite{KZ2010} to get
\begin{equation}\label{ZaZa_BOUND}
{\sf Z}_{\ssy A}(t)\leq\,C\,\Delta{t}^{\frac{3}{8}}
\quad\forall\,t\in(0,T].
\end{equation}
\par
The error bound \eqref{ModelError} follows by
observing that ${\sf Z}(0)=0$ and combining the bounds
\eqref{corv1}, \eqref{ZaZa_BOUND} and \eqref{ZbZb_bound}.
\end{proof}
%
%
%
%
%
%
%
%
\section{Time-Discrete Approximations}\label{SECTION3}
\subsection{The deterministic problem}\label{section3A}
In this section we introduce and analyze modified Crank-Nicolson time-discrete approximations,
$(W^m)_{m=0}^{\ssy M}$,  of the  solution $w$ to the deterministic problem \eqref{Det_Parab}.
%
%
\par
We begin by setting
\begin{equation}\label{CNDet1}
W^0:=w_0
\end{equation}
and then by finding $W^1\in{\bfdot H}^4(D)$ such that
\begin{equation}\label{CNDet12}
W^1-W^{0} +\tfrac{\Delta\tau}{2}\,\partial_x^4W^1=0.
\end{equation}
Finally, for $m=2,\dots,M$, we specify $W^m\in{\bfdot H}^4(D)$ such that
\begin{equation}\label{CNDet2}
W^m-W^{m-1} +\Delta\tau\,\partial_x^4W^{m-\frac{1}{2}}=0.
\end{equation}
%
%
\par
First, we provide a discrete in time $L^2_t(L^2_x)$
a priori estimate of time averages of the nodal error for the modified Crank-Nicolson
time-discrete approximations defined above.
\begin{proposition}\label{DetPropo1}
Let $(W^m)_{m=0}^{\ssy M}$ be the modified Crank-Nicolson time-discrete
approximations of the solution $w$ to the problem \eqref{Det_Parab}
defined by \eqref{CNDet1}, \eqref{CNDet12} and \eqref{CNDet2}.
Then, there exists a constant $C>0$, independent of $\dtau$, such that
\begin{equation}\label{May2015_0}
\left(\,\dtau\,\sum_{m=1}^{\ssy M}\|W^{m-\half}-w^{m-\half}\|_{\ssy
0,D}^2\right)^{\frac{1}{2}} \leq\,C\,\Delta\tau^{\theta}
\,\|w_0\|_{\ssy{\bfdot H}^{4\theta-2}}\quad\forall\,\theta\in[0,1],
\quad\forall\,w_0\in{\bfdot H}^2(D),
\end{equation}
where $w^{\ell}(\cdot):=w(\tau_{\ell},\cdot)$ for $\ell=0,\dots,M$.
\end{proposition}
%
%
%
%
%
%
%
%
\begin{proof}
The error bound \eqref{May2015_0} follows by interpolation after proving it for $\theta=1$ and $\theta=0$ 
(cf. \cite{BinLi}, \cite{KZ2010}, \cite{YubinY03}).  In the sequel,  we will use the symbol
$C$ to denote a generic constant that is independent of $\dtau$ and may changes value from
one line to the other.
\par\noindent\vskip0.2truecm\par
\textbullet\,{\tt Case $\theta=1$:}
Let ${\sf E}^{\star}:=w(\tau_{\half},\cdot)-W^1$ and
${\sf E}^m:=w^m-W^m$ for $m=0,\dots,M$. Using \eqref{Det_Parab}
and \eqref{CNDet2}, we arrive at
\begin{equation}\label{May2015_1}
T_{\ssy B}({\sf E}^m-{\sf E}^{m-1})+\Delta\tau\,{\sf E}^{m-\frac{1}{2}}
=\sigma_m,\quad m=2,\dots,M,
\end{equation}
where
\begin{equation}\label{May2015_1a}
\sigma_{\ell}(\cdot):=-\int_{\ssy\Delta_{\ell}}
\left[w(\tau,\cdot)-w^{\ell-\frac{1}{2}}(\cdot)\right]\,d\tau,
\quad\ell=2,\dots,M.
\end{equation}
Taking the $L^2(D)-$inner product of both sides
of \eqref{May2015_1} with ${\sf E}^{m-\half}$, and then using
\eqref{TB-prop1} and summing with respect to $m$,
from $2$ up to $M$, we obtain
\begin{equation}\label{May2015_2}
\|T_{\ssy E}{\sf E}^{\ssy M}\|^2_{\ssy 0,D}
-\|T_{\ssy E}{\sf E}^{1}\|^2_{\ssy 0,D}
+2\,\Delta\tau\,\sum_{m=2}^{\ssy M}\|{\sf E}^{m-\half}\|_{\ssy 0,D}^2
=2\,\sum_{m=2}^{\ssy M}(\sigma_m,{\sf E}^{m-\half})_{\ssy 0,D}.
\end{equation}
Applying the Cauchy-Schwarz inequality and the geometric mean inequality,
we have
\begin{equation*}
\begin{split}
2\sum_{m=2}^{\ssy M}(\sigma_m,{\sf E}^{m-\half})_{\ssy 0,D}
\leq&\,2\sum_{m=2}^{\ssy M}\|\sigma_m\|_{\ssy 0,D}
\,\|{\sf E}^{m-\half}\|_{\ssy 0,D}\\
\leq&\,\sum_{m=2}^{\ssy M}\left(\,\Delta\tau^{-1}\|\sigma_m\|_{\ssy 0,D}^2
+\Delta\tau\,\|{\sf E}^{m-\half}\|^2_{\ssy 0,D}\,\right),\\
\end{split}
\end{equation*}
which, along with \eqref{May2015_2}, yields
\begin{equation}\label{May2015_3}
\dtau\,\sum_{m=2}^{\ssy M}\|{\sf E}^{m-\half}\|_{\ssy 0,D}^2
\leq\,\|T_{\ssy E}{\sf E}^{1}\|^2_{\ssy 0,D}
+\dtau^{-1}\,\sum_{m=2}^{\ssy M}\|\sigma_m\|_{\ssy 0,D}^2.
\end{equation}
Using \eqref{May2015_1a}, we bound the quantities $(\sigma_m)_{m=2}^{\ssy M}$
as follows:
\begin{equation}\label{May2015_4}
\begin{split}
\|\sigma_m\|_{\ssy 0,D}^2=&\,\tfrac{1}{4}\,\int_{\ssy D}
\left(-\int_{\ssy\Delta_m}\!\!\int_{\tau}^{\tau_m}
\partial_{\tau}w(s,x)\,ds{d\tau}
+\int_{\ssy\Delta_m}\!\!\int_{\tau_{m-1}}^{\tau}
\partial_{\tau}w(s,x)\,ds{d\tau}\right)^2\,dx\\
\leq&\,\int_{\ssy D}\left(\int_{\ssy\Delta_m}
\!\!\int_{\ssy\Delta_m}|\partial_{\tau}w(s,x)|\,ds
d\tau\right)^2\,dx\\
\leq&\,\Delta\tau^3\,\int_{\ssy\Delta_m}
\|\partial_{\tau}w(s,\cdot)\|_{\ssy 0,D}^2\,ds,
\quad m=2,\dots,M.\\
\end{split}
\end{equation}
Using that ${\sf E}^0=0$ and combining \eqref{May2015_3},
\eqref{May2015_4} and \eqref{Reggo3} (with $\beta=0$, $\ell=1$, $r=0$),
we obtain
\begin{equation}\label{May2015_5}
\begin{split}
\dtau\,\sum_{m=1}^{\ssy M}\|{\sf E}^{m-\frac{1}{2}}\|_{\ssy
0,D}^2\leq&\,\tfrac{\dtau}{4}\,\|{\sf E}^1\|_{\ssy 0,D}^2
+\|T_{\ssy E}{\sf E}^{1}\|^2_{\ssy 0,D}
+\Delta\tau^2\,
\int_0^{\ssy T}\|\partial_{\tau}w(s,\cdot)\|_{\ssy 0,D}^2\,ds\\
\leq&\,\tfrac{\dtau}{4}\,\|{\sf E}^1\|_{\ssy 0,D}^2
+\|T_{\ssy E}{\sf E}^{1}\|^2_{\ssy 0,D}
+C\,\Delta\tau^2\,\|w_0\|_{\ssy{\bfdot H}^2}^2.\\
\end{split}
\end{equation}
In order to bound the first two terms in the right hand side of
\eqref{May2015_5}, we introduce the following splittings
\begin{equation}\label{May2015_6}
\|T_{\ssy E}{\sf E}^1\|_{\ssy 0,D}^2\leq\,2\,\left(\,
\|T_{\ssy E}(w(\tau_1,\cdot)-w(\tau_{\half},\cdot))\|^2_{\ssy 0,D}
+\|T_{\ssy E}{\sf E}^{\star}\|^2_{\ssy 0,D}\,\right)
\end{equation}
and
\begin{equation}\label{May2015_6a}
\dtau\,\|{\sf E}^1\|_{\ssy 0,D}^2\leq\,2\,\dtau\,\left(\,
\|w(\tau_1,\cdot)-w(\tau_{\half},\cdot)\|_{\ssy 0,D}^2
+\|{\sf E}^{\star}\|_{\ssy 0,D}^2\,\right).
\end{equation}
We continue by estimating the terms in the right hand side of \eqref{May2015_6}
and \eqref{May2015_6a}. First, we observe that
\begin{equation*}
\begin{split}
\|w(\tau_1,\cdot)-w(\tau_{\half},\cdot)\|_{\ssy 0,D}^2
=&\,\int_{\ssy D}\bigg(\int_{\tau_{\half}}^{\tau_1}
\partial_{\tau}w(\tau,x)\;d\tau\bigg)^2\;dx\\
\leq&\,\tfrac{\dtau}{2}\,\int_{\tau_{\half}}^{\tau_1}
\|\partial_{\tau}w(\tau,\cdot)\|^2_{\ssy 0,D}\;d\tau,\\
\end{split}
\end{equation*}
which, along with \eqref{Reggo3} (with $\ell=1$, $r=0$, $\beta=0$), yields 
\begin{equation}\label{May2015_7}
\|w(\tau_1,\cdot)-w(\tau_{\half},\cdot)\|_{\ssy 0,D}^2
\leq\,C\,\dtau\,\|w_0\|^2_{\ssy{\bfdot H}^2}.
\end{equation}
Next, we use \eqref{ElReg1} and \eqref{minus_equiv}, to get
\begin{equation*}\label{May2015_8a}
\begin{split}
\|T_{\ssy E}(w(\tau_1,\cdot)-w(\tau_{\half},\cdot))\|_{\ssy 0,D}^2
=&\,\int_{\ssy D}\,\bigg(\int_{\tau_{\half}}^{\tau_1}
T_{\ssy E}(\partial_{\tau}w(\tau,x))\;d\tau\bigg)^2\;dx\\
\leq&\,\tfrac{\dtau}{2}\,\int_{\tau_{\half}}^{\tau_1}
\|T_{\ssy E}\left(\partial_{\tau}w(\tau,\cdot)\right)\|_{\ssy 0,D}^2\;d\tau\\
\leq&\,C\,\dtau\,\int_{\tau_{\half}}^{\tau_1}
\|\partial_{\tau}w(\tau,\cdot)\|_{\ssy -2,D}^2\;d\tau\\
\leq&\,C\,\Delta\tau\,\int_{\tau_{\half}}^{\tau_1}
\|\partial_{\tau}w(\tau,\cdot)\|^2_{\ssy{\bfdot H}^{-2}}\;d\tau.\\
\end{split}
\end{equation*}
Observing that
\begin{equation}\label{Exx_Sol}
w(\tau,\cdot)=\sum_{k=1}^{\infty}e^{-\lambda_k^4\,\tau}
\,(w_0,\varepsilon_k)_{\ssy 0,D}\,\varepsilon_k(\cdot)
\quad\forall\,\tau\in[0,T],
\end{equation}
we have
\begin{equation*}\label{May2015_8b}
\begin{split}
\|\partial_{\tau}w(\tau,\cdot)\|_{\ssy{\bfdot H}^{-2}}^2
=&\,\sum_{k=1}^{\infty}\lambda_k^{-4}
\,\,|(\partial_{\tau}w(\tau,\cdot),\varepsilon_k)_{\ssy 0,D}|^2\\
=&\,\sum_{k=1}^{\infty}\lambda_k^4\,e^{-2\,\lambda_k^4\,\tau}
\,|(w_0,\varepsilon_k)_{\ssy 0,D}|^2\\
\leq&\,\|w_0\|_{\ssy{\bfdot H}^2}^2\quad\forall\,\tau\in[0,T].\\
\end{split}
\end{equation*}
Thus,  we arrive at
\begin{equation}\label{May2015_9}
\|T_{\ssy E}(w(\tau_1,\cdot)-w(\tau_{\half},\cdot))\|_{\ssy 0,D}^2\leq\,C\,
\Delta\tau^2\,\|w_0\|_{\ssy{\bfdot H}^2}^2.
\end{equation}
Finally, using \eqref{Det_Parab} and \eqref{CNDet12} we have
\begin{equation}\label{May2015_10}
T_{\ssy B}({\sf E}^{\star}-{\sf E}^0)+\tfrac{\Delta\tau}{2}\,{\sf E}^{\star}=\sigma_{\star}
\end{equation}
with 
\begin{equation}\label{May2015_11}
\sigma_{\star}(\cdot):=-\int_{0}^{\tau_{\frac{1}{2}}}\left[\,w(s,\cdot)-w(\tau_{\half},\cdot)\right]\;ds.
\end{equation}
Since ${\sf E}^0=0$, after taking the $L^2(D)-$inner product of both sides
of \eqref{May2015_10} with ${\sf E}^{\star}$ and using \eqref{TB-prop1} and 
the Cauchy-Schwarz inequality along with the arithmetic mean inequality, we obtain
\begin{equation}\label{May2015_12}
\begin{split}
\|T_{\ssy E}{\sf E}^{\star}\|_{\ssy 0,D}^2+\tfrac{\dtau}{2}\,\|{\sf E}^{\star}\|_{\ssy 0,D}^2
=&\,(\sigma_{\star},{\sf E}^{\star})_{\ssy 0,D}\\
\leq&\,\tfrac{1}{\dtau}\,\|\sigma_{\star}\|_{\ssy 0,D}^2
+\tfrac{\dtau}{4}\,\|{\sf E}^{\star}\|_{\ssy 0,D}^2.\\
\end{split}
\end{equation}
Now, using \eqref{May2015_11} and \eqref{Reggo3} (with $\beta=0$, $\ell=1$, $r=0$)
we obtain
\begin{equation*}
\begin{split}
\|\sigma_{\star}\|_{\ssy 0,D}^2=&\,\int_{\ssy D}
\left[\,\int_0^{\tau_{\half}}\left(\int_s^{\tau_{\half}}\partial_{\tau}w(\tau,x)\;d\tau\right)\;ds
\,\right]^2\;dx\\
\leq&\,\tfrac{\Delta\tau^2}{4}\int_{\ssy D}
\left(\,\int_0^{\tau_{\frac{1}{2}}}|\partial_{\tau}w(\tau,x)|\;d\tau\,\right)^2\;dx\\
\leq&\,\tfrac{\Delta\tau^3}{8}
\int_0^{\tau_1}\|\partial_{\tau}w(\tau,\cdot)\|^2_{\ssy 0,D}\;d\tau\\
\leq&\,C\,\Delta\tau^3\,\|w_0\|_{\ssy{\bfdot H}^2}^2,\\
\end{split}
\end{equation*}
which, along with \eqref{May2015_12}, yields
\begin{equation}\label{May2015_13}
\|T_{\ssy E}{\sf E}^{\star}\|_{\ssy 0,D}^2+\tfrac{\dtau}{4}\,\|{\sf E}^{\star}\|_{\ssy 0,D}^2
\leq\,C\,\dtau^2\,\|w_0\|^2_{\ssy{\bfdot H}^2}.
\end{equation}
Thus, from  \eqref{May2015_5}, \eqref{May2015_6}, \eqref{May2015_6a},
\eqref{May2015_7}, \eqref{May2015_9} and \eqref{May2015_13}, we conclude that
\begin{equation*}\label{May2015_14}
\dtau\,\sum_{m=1}^{\ssy M}\|{\sf E}^{m-\frac{1}{2}}\|_{\ssy
0,D}^2\leq\,C\,\dtau^2\,\|w_0\|^2_{\ssy{\bfdot H}^2}.
\end{equation*}
\par\noindent\vskip0.2truecm\par
\textbullet\,{\tt Case $\theta=0$:}
First, we observe that \eqref{CNDet12} and \eqref{CNDet2} are equivalent
to
\begin{equation}\label{May2015_20a}
T_{\ssy B}(W^1-W^{0})+\tfrac{\dtau}{2}\,W^{1}=0
\end{equation}
and
\begin{equation}\label{May2015_20}
T_{\ssy B}(W^m-W^{m-1})+\dtau\,W^{m-\frac{1}{2}}=0,\quad
m=2,\dots,M.
\end{equation}
Next, we take the $L^2(D)-$inner product of both sides of
\eqref{May2015_20} with $W^{m-\frac{1}{2}}$,
use \eqref{TB-prop1} and sum with respect to $m$ from
$2$ up to $M$, to obtain
\begin{equation*}
\|T_{\ssy E}W^{\ssy M}\|_{\ssy 0,D}^2-\|T_{\ssy E}W^1\|_{\ssy 0,D}^2
+2\,\dtau\,\sum_{m=2}^{\ssy M}\|W^{m-\frac{1}{2}}\|_{\ssy 0,D}^2=0,
\end{equation*}
which yields that
\begin{equation}\label{May2015_21}
\dtau\,\|W^1\|_{\ssy 0,D}^2
+\dtau\,\sum_{m=2}^{\ssy M}\|W^{m-\frac{1}{2}}\|_{\ssy 0,D}^2\leq
\,\dtau\,\|W^1\|_{\ssy 0,D}^2+\|T_{\ssy E}W^1\|_{\ssy 0,D}^2.
\end{equation}
Now, we take the $L^2(D)-$inner product of both sides of
\eqref{May2015_20a} with $W^{1}$,
use \eqref{TB-prop1} along with \eqref{innerproduct} to get
\begin{equation}\label{May2015_22}
\|T_{\ssy E}W^1\|_{\ssy 0,D}^2+\dtau\,\|W^1\|_{\ssy 0,D}^2
\leq\,\|T_{\ssy E}W^0\|_{\ssy 0,D}^2.
\end{equation}
Combining \eqref{May2015_21} and \eqref{May2015_22} and then
using \eqref{ElReg1} and \eqref{minus_equiv}, we obtain
\begin{equation}\label{May2015_23}
\begin{split}
\dtau\,\|W^1\|_{\ssy 0,D}^2
+\dtau\,\sum_{m=2}^{\ssy M}\|W^{m-\frac{1}{2}}\|_{\ssy 0,D}^2
\leq&\,\|T_{\ssy E}w_0\|_{\ssy 0,D}^2\\
\leq&\,C\,\|w_0\|_{\ssy -2,D}^2\\
\leq&\,C\,\|w_0\|_{\ssy{{\bfdot H}^{-2}}}^2.\\
\end{split}
\end{equation}
In addition, we have
\begin{equation}\label{Museum_1}
\dtau\,\|w^1\|^2_{\ssy 0,D}+\dtau\sum_{m=2}^{\ssy M}\|w^{m-\half}\|_{\ssy 0,D}^2
\leq\,2\,\dtau\,\sum_{m=1}^{\ssy M}\|w^m\|_{\ssy 0,D}^2
\end{equation}
and
\begin{equation*}
\begin{split}
2\,\dtau\,\sum_{m=1}^{\ssy M}\|w^m\|_{\ssy 0,D}^2\leq&\,2\,\dtau^{-1}\,\sum_{m=1}^{\ssy M}\int_{\ssy D}
\left(\,\int_{\tau_{m-1}}^{\tau_m}
\partial_{\tau}\left[\,(\tau-\tau_{m-1})
\,w(\tau,x)\,\right]\,d\tau\,\right)^2dx\\
\leq&\,2\,\dtau^{-1}\,\sum_{m=1}^{\ssy M}\int_{\ssy D}
\left(\,\int_{\tau_{m-1}}^{\tau_m}\left[\,w(\tau,x)
+(\tau-\tau_{m-1})\,w_{\tau}(\tau,x)\,\right]
\,d\tau\,\right)^2dx\\
\leq&\,4\,\sum_{m=1}^{\ssy M}\int_{\ssy\Delta_m}
\left(\,\|w(\tau,\cdot)\|_{\ssy 0,D}^2 +(\tau-\tau_{m-1})^2\,
\|w_{\tau}(\tau,\cdot)\|_{\ssy 0,D}^2\,\right)\;d\tau\\
\leq&\,4\,\int_0^{\ssy T}\left(\,\|w(\tau,\cdot)\|_{\ssy 0,D}^2
+\tau^2\,\|w_{\tau}(\tau,\cdot)\|_{\ssy 0,D}^2\,\right)\;d\tau,
\end{split}
\end{equation*}
which, along with \eqref{Reggo3} (taking $(\beta,\ell,r)=(0,0,0)$
and $(\beta,\ell,r)=(2,1,0)$), yields
\begin{equation}\label{May2015_24}
2\,\dtau\,\sum_{m=1}^{\ssy M}\|w^m\|_{\ssy 0,D}^2\leq\,C\,\|w_0\|_{\ssy {\bfdot H}^{-2}}^2.
\end{equation}
Observing that  ${\sf E}^{\half}=\half\,{\sf E}^1$, we have
\begin{equation*}
\begin{split}
\dtau\,\sum_{m=1}^{\ssy M}\|{\sf E}^{m-\half}\|_{\ssy 0,D}^2
\leq&\,2\,\left(\,
\dtau\,\|W^1\|_{\ssy 0,D}^2+\dtau\,\sum_{m=2}^{\ssy M}
\|W^{m-\half}\|_{\ssy 0,D}^2\,\right)\\
&\,+2\, \left(\,
\dtau\,\|w^1\|_{\ssy 0,D}^2+\dtau\,\sum_{m=2}^{\ssy M}\|w^{m-\half}\|_{\ssy 0,D}^2\,\right),\\
\end{split}
\end{equation*}
which, after using \eqref{May2015_23}, \eqref{Museum_1}
and \eqref{May2015_24}, yields
\begin{equation*}
\sum_{m=1}^{\ssy M}\Delta\tau\,\|{\sf E}^{m-\half}\|_{\ssy
0,D}^2\leq\,C\,\|w_0\|_{\ssy {\bfdot H}^{-2}}^2.
\end{equation*}
\end{proof}
\par
Next, we show a discrete in time $L^2_t(L^2_x)$ a priori estimate of the nodal error for the
modified Crank-Nicolson time-discrete approximations.
\begin{proposition}\label{Propo_Boom1}
Let $(W^m)_{m=0}^{\ssy M}$ be the modified Crank-Nicolson time-discrete
approximations of the solution $w$ to the problem \eqref{Det_Parab}
defined by \eqref{CNDet1}, \eqref{CNDet12} and \eqref{CNDet2}.
Then, there exists a constant $C>0$, independent of $\dtau$, such that
\begin{equation}\label{September2015_0}
\left(\dtau\,\sum_{m=1}^{\ssy M}\|W^m-w^m\|_{\ssy 0,D}^2\right)^{\frac{1}{2}}
\leq\,C\,\Delta\tau^{\frac{\delta}{2}}
\,\|w_0\|_{\ssy{\bfdot H}^{2(\delta-1)}}\quad\forall\,\delta\in[0,1],
\quad\forall\,w_0\in{\bfdot H}^2(D),
\end{equation}
where $w^{\ell}:=w(\tau_{\ell},\cdot)$ for $\ell=0,\dots,M$.
\end{proposition}
\begin{proof}
We will arrive at the error bound \eqref{September2015_0} by interpolation after
proving it for $\delta=1$ and $\delta=0$  (cf. Proposition~\ref{DetPropo1}).
In both cases, the error estimation is based on the following bound
\begin{equation}\label{Hadji0}
\left(\dtau\sum_{m=1}^{\ssy M}\|W^m-w^m\|_{\ssy 0,D}^2\right)^{\frac{1}{2}}\leq
S_{\ssy B}+S_{\ssy C}+S_{\ssy D}
\end{equation}
where
\begin{equation*}
\begin{split}
S_{\ssy B}:=&\,\left(\dtau\sum_{m=2}^{\ssy M}\|W^m-W^{m-\half}\|_{\ssy 0,D}^2\right)^{\frac{1}{2}},\\
S_{\ssy C}:=&\,\left(\dtau\,\|W^1-w^1\|_{\ssy 0,D}^2
+\dtau\sum_{m=2}^{\ssy M}\|W^{m-\half}-w^{m-\half}\|_{\ssy 0,D}^2\right)^{\frac{1}{2}},\\
S_{\ssy D}:=&\,\left(\dtau\sum_{m=2}^{\ssy M}\|w^{m-\half}-w^m\|_{\ssy 0,D}^2\right)^{\frac{1}{2}}.\\
\end{split}
\end{equation*}
In the sequel,  we will use the symbol
$C$ to denote a generic constant that is independent of $\dtau$ and may changes value from
one line to the other.
\par\noindent\vskip0.2truecm\par
\textbullet\,{\tt Case $\delta=1$:}
Taking the $L^2(D)-$inner product of both sides of \eqref{CNDet2} with $(W^m-W^{m-1})$
and then integrating by parts, we easily arrive at
\begin{equation}\label{Hadji1}
\|W^m-W^{m-1}\|^2_{\ssy 0,D}+\tfrac{\dtau}{2}\,\left(\,\|\partial_x^2W^m\|_{\ssy 0,D}^2
-\|\partial_x^2W^{m-1}\|_{\ssy 0,D}^2\,\right)=0,\quad m=2,\dots,M.
\end{equation}
After summing both sides of \eqref{Hadji1} with respect to $m$ from $2$ up to $M$, we obtain
\begin{equation*}
\dtau\sum_{m=2}^{\ssy M}\|W^m-W^{m-1}\|^2_{\ssy 0,D}
+\tfrac{\dtau^2}{2}\,\left(\,\|\partial_x^2W^{\ssy M}\|_{\ssy 0,D}^2
-\|\partial_x^2W^{1}\|_{\ssy 0,D}^2\,\right)=0,
\end{equation*}
which yields
\begin{equation}\label{Hadji2}
\dtau\sum_{m=2}^{\ssy M}\|W^m-W^{m-1}\|^2_{\ssy 0,D}
\leq\tfrac{\dtau^2}{2}\,\|\partial_x^2W^{1}\|_{\ssy 0,D}^2.
\end{equation}
Taking the $L^2(D)-$inner product of both sides of \eqref{CNDet12} with $W^1$,
and then integrating by parts and using \eqref{innerproduct}, we have
\begin{equation*}
\|W^1\|_{\ssy 0,D}^2-\|W^0\|_{\ssy 0,D}^2+\dtau\,\|\partial_x^2W^1\|_{\ssy 0,D}^2\leq0
\end{equation*}
from which we conclude that
\begin{equation}\label{Hadji3}
\dtau\,\|\partial_x^2W^1\|_{\ssy 0,D}^2\leq\|w_0\|_{\ssy 0,D}^2.
\end{equation}
Thus, combining \eqref{Hadji2} and \eqref{Hadji3}, we get
\begin{equation}\label{Hadji4}
\begin{split}
S_{\ssy B}=&\,\tfrac{1}{2}\,
\left(\dtau\sum_{m=2}^{\ssy M}\|W^m-W^{m-1}\|_{\ssy 0,D}^2\right)^{\half}\\
\leq&\,\tfrac{1}{2\sqrt{2}}\,\dtau\,\|\partial_x^2W^1\|_{\ssy 0,D}\\
\leq&\,\dtau^{\half}\,\|w_0\|_{\ssy 0,D}.\\
\end{split}
\end{equation}
Also, we observe that the estimate \eqref{May2015_0}, for $\theta=\frac{1}{2}$, yields
\begin{equation}\label{Hadji44}
S_{\ssy C}\leq\,C\,\dtau^{\frac{1}{2}}\,\|w_0\|_{\ssy 0,D}.
\end{equation}
Finally, using \eqref{Reggo0} (with $\ell=1$, $r=0$, $q=0$), we obtain
\begin{equation}\label{Hadji45}
\begin{split}
S_{\ssy D}\leq&\,\left(\dtau\sum_{m=2}^{\ssy M}
\left\|\int_{\ssy\Delta_m}\partial_{\tau}w(\tau,\cdot)\;d\tau\right\|_{\ssy 0,D}^2
\right)^{\frac{1}{2}}\\
\leq&\,\left(\dtau^2\int_{\dtau}^{\ssy T}\|\partial_{\tau}w(\tau,\cdot)\|^2_{\ssy 0,D}\;d\tau
\right)^{\frac{1}{2}}\\
\leq&\,C\,\left(\dtau^2\int_{\dtau}^{\ssy T}\tau^{-2}\,\|w_0\|^2_{\ssy 0,D}\;d\tau
\right)^{\frac{1}{2}}\\
\leq&\,C\,\dtau\,\|w_0\|_{\ssy 0,D}\,\left(\tfrac{1}{\dtau}-\tfrac{1}{T}\right)^{\frac{1}{2}}\\
\leq&\,C\,\dtau^{\half}\,\|w_0\|_{\ssy 0,D}.\\
\end{split}
\end{equation}
Thus, from \eqref{Hadji0}, \eqref{Hadji4}, \eqref{Hadji44} and \eqref{Hadji45} we conclude
\eqref{September2015_0} for $\delta=1$.
\par\noindent\vskip0.2truecm\par
\textbullet\,{\tt Case $\delta=0$:}
Taking again the $L^2(D)-$inner product of both sides of \eqref{CNDet12} with $W^1$
and then integrating by parts and using \eqref{minus_equiv} and \eqref{H_equiv} along with the
arithmetic mean inequality,  we obtain
\begin{equation}\label{Hadji5}
\begin{split}
\|W^1\|_{\ssy 0,D}^2+\tfrac{\dtau}{2}\,\|\partial_x^2W^1\|_{\ssy 0,D}^2=&\,(w_0,W^1)_{\ssy 0,D}\\
\leq&\,\|w_0\|_{\ssy -2,D}\,\|W^1\|_{\ssy 2,D}\\
\leq&\,C\,\|w_0\|_{\ssy{\bfdot H}^{-2}}\,\|W^1\|_{\ssy{\bfdot H}^2}\\
\leq&\,C\,\dtau^{-1}\|w_0\|^2_{\ssy{\bfdot H}^{-2}}
+\tfrac{\dtau}{4}\,\|W^1\|_{\ssy{\bfdot H}^2}^2.
\end{split}
\end{equation}
Now, integrating by parts we have
\begin{equation}\label{X_evid}
\begin{split}
\|\partial_x^2W^1\|_{\ssy 0,D}=&\,\left(\sum_{k=1}^{\infty}|(\varepsilon_k,\partial_x^2W^1)_{\ssy 0,D}|^2\right)^{\half}\\
=&\,\left(\sum_{k=1}^{\infty}|(\partial_x^2\varepsilon_k,W^1)_{\ssy 0,D}|^2\right)^{\half}\\
=&\,\left(\sum_{k=1}^{\infty}\lambda_k^4\,|(\varepsilon_k,W^1)_{\ssy 0,D}|^2\right)^{\half}\\
=&\,\|W^1\|_{\ssy {\bfdot H}^2},\\
\end{split}
\end{equation}
which, along with \eqref{Hadji5}, yields that
\begin{equation}\label{Hadji6}
\dtau^2\,\|\partial_x^2W^1\|_{\ssy 0,D}^2\leq\,C\,\|w_0\|_{\ssy{{\bfdot H}^{-2}}}^2.
\end{equation}
Thus, combining \eqref{Hadji2} and \eqref{Hadji6}, we conclude that
\begin{equation}\label{Hadji8}
\begin{split}
S_{\ssy B}=&\,\tfrac{1}{2}\,\left(\dtau
\sum_{m=2}^{\ssy M}\|W^m-W^{m-1}\|_{\ssy 0,D}^2\right)^{\half}\\
\leq&\,\tfrac{1}{2\sqrt{2}}\,\dtau\,\|\partial_x^2W^1\|_{\ssy 0,D}\\
\leq&\,C\,\|w_0\|_{\ssy{{\bfdot H}^{-2}}}.
\end{split}
\end{equation}
Also, the estimate \eqref{May2015_0}, for $\theta=0$, yields
\begin{equation}\label{Hadji11}
S_{\ssy C}\leq\,C\,\|w_0\|_{\ssy{{\bfdot H}^{-2}}}.
\end{equation}
Using the Cauchy-Schwarz inequality and \eqref{May2015_24}, we have
\begin{equation}\label{Hadji12}
\begin{split}
S_{\ssy D}=&\,\tfrac{1}{2}
\,\left(\dtau\sum_{m=2}^{\ssy M}\|w^m-w^{m-1}\|_{\ssy 0,D}^2\right)^{\half}\\
\leq&\,\tfrac{\sqrt{2}}{2}
\,\left(\dtau\sum_{m=1}^{\ssy M}\|w^m\|_{\ssy 0,D}^2\right)^{\half}\\
\leq&\,C\,\|w_0\|_{\ssy {\bfdot H}^{-2}}^2.\\
\end{split}
\end{equation}
Thus, from \eqref{Hadji0}, \eqref{Hadji8}, \eqref{Hadji11} and \eqref{Hadji12} we conclude
\eqref{September2015_0} for $\delta=0$.
\end{proof}
%
\subsection{The stochastic problem}\label{CN_Main_2}
The Crank-Nicolson time-stepping method for the approximate
problem \eqref{AC2} constructs, for $m=0,\dots,M$, an
approximation $U^m$ of ${\widehat u}(\tau_m,\cdot)$, first by
setting
\begin{equation}\label{CN_S1}
U^0:=0,
\end{equation}
and then, for $m=1,\dots,M$, by specifying
$U^m\in{\bfdot H}^4(D)$ such that
\begin{equation}\label{CN_S2}
U^m-U^{m-1}+\dtau\,\partial_x^4U^{m-\frac{1}{2}}
=\int_{\ssy\Delta_m}{\widehat W}\,ds\quad\text{\rm a.s.}.
\end{equation}
\par
In the theorem that follows, using the result of Proposition~\ref{Propo_Boom1} we show
a discrete in time $L^{\infty}_t(L^2_{\ssy P}(L^2_x))$ convergence estimate for the
Crank-Nicolson  time discrete approximations of ${\widehat u}$ defined above.
%
%
%
%
\begin{theorem}\label{TimeDiscreteErr1}
Let ${\widehat u}$ be the solution of \eqref{AC2} and
$(U^m)_{m=0}^{\ssy M}$ be the Crank-Nicolson
time-discrete appro\-ximations of ${\widehat u}$ specified in \eqref{CN_S1} and \eqref{CN_S2}.
Then, there exists constant $C>0$, independent of $\Delta{t}$, ${\sf M}_{\star}$
and $\dtau$, such that
\begin{equation}\label{Dinner0}
\max_{1\leq m \leq {\ssy M}} \left(\,{\mathbb E}
\left[ \|U^m-{\widehat u}^m\|_{\ssy
0,D}^2\right]\,\right)^{\half} \leq\,C
\,\,\,\epsilon^{-\half}\,\,\,\dtau^{\frac{3}{8}-\epsilon}
\quad\forall\,\epsilon\in\left(0,\tfrac{3}{8}\right]
\end{equation}
where ${\widehat u}^{\ell}:={\widehat u}(\tau_{\ell},\cdot)$ for $\ell=0,\dots,M$.
\end{theorem}
%
%
%
%
%
%
%
%
%
%
\begin{proof}
Let ${\sf I}:L^2(D)\to L^2(D)$ be the identity operator,
${\sf Y}:H^4(D)\rightarrow L^2(D)$ be defined by
${\sf Y}:={\sf I}-\tfrac{\Delta\tau}{2}\,\partial_x^4$
and $\Lambda:L^2(D)\to{\bfdot H}^4(D)$ be the inverse
elliptic operator $\Lambda:=({\sf I}+\tfrac{\Delta{\tau}}{2}\,\partial_x^4)^{-1}$.
Finally, for $m=1,\dots,M$, we define an operator
${\mathcal Q}_m:L^2(D)\to{\bfdot H}^4(D)$ by
${\mathcal Q}_m:=(\Lambda{\sf Y})^{m-1}\Lambda$.
The operator $\Lambda$ has Green function
${\sf G}_{\ssy\Lambda}(x,y)=
\sum_{k=1}^{\infty}\frac{\varepsilon_k(x)
\,\varepsilon_k(y)}{1+\frac{\Delta\tau}{2}\lambda_k^4}$, i.e.
$\Lambda{f}(x)=\int_{\ssy D}{\sf G}_{\ssy\Lambda}(x,y)f(y)\,dy$
for $x\in{\overline D}$ and $f\in L^2(D)$. Also, ${\sf Y}$ has
Green function
${\sf G}_{\ssy{\sf Y}}(x,y)=\sum_{k=1}^{\infty}(1-\frac{\Delta\tau}{2}\lambda_k^4)
\,\varepsilon_k(x)\,\varepsilon_k(y)$, i.e.
${\sf Y}z(x)=\int_{\ssy D}{\sf G}_{\ssy{\sf Y}}(x,y)z(y)\;dy$ for
$x\in{\overline D}$ and $z\in H^4(D)$. Finally, for $m=1,\dots,M$,
${\mathcal Q}_m$ has Green function ${\sf G}_{\ssy{\mathcal Q}_m}$
given by
\begin{equation*}
{\sf G}_{\ssy{\mathcal Q}_m}=
\sum_{k=1}^{\infty}\tfrac{(1-\frac{\Delta\tau}{2}\lambda_k^4)^{m-1}}
{(1+\frac{\Delta\tau}{2}\lambda_k^4)^{m}}\,\varepsilon_k(x)
\,\varepsilon_k(y),
\end{equation*}
i.e. ${\mathcal Q}_mf(x)=\int_{\ssy D}{\sf G}_{\ssy{\mathcal Q}_m}(x,y)f(y)\;dy$ for
$x\in{\overline D}$ and $f\in L^2(D)$.
\par
For a given $w_0\in{\bfdot H}^2(D)$,
let $(W^m)_{m=0}^{\ssy M}$ be the modified Crank-Nicolson time-discrete
approximations defined by \eqref{CNDet1}, \eqref{CNDet12} and \eqref{CNDet2}. 
Then, using a simple induction argument, we conclude that
\begin{equation}\label{EMP_2010_1}
W^m={\mathcal Q}_{m}w_0,\quad m=1,\dots,M.
\end{equation}
%
%
Also, to simplify the notation, we set
$G_m(\tau;x,y):={\mathfrak X}_{(0,\tau_m)}(\tau)
\,\,G(\tau_m-\tau;x,y)$ for $m=1,\dots,M$.
\par
In the sequel, we will use the symbol $C$ to denote a generic constant that is independent
of $\Delta{t}$, ${\sf M}_{\star}$ and $\dtau$, and may changes value from one line to the other.
%
\par
Using \eqref{CN_S2} and an induction argument, we conclude that
\begin{equation*}
U^m=\sum_{\ell=1}^{\ssy m}
\int_{\ssy\Delta_{\ell}}{\mathcal Q}_{m-\ell+1}\left({\widehat
W}(\tau,\cdot)\right)\,\,d\tau, \quad m=1,\dots,M,
\end{equation*}
which yields
\begin{equation}\label{Dinner1}
U^m(x) =\int_0^{\ssy T}\!\!\!\int_{\ssy D}
{\mathcal K}_m(\tau;x,y)\,{\widehat W}(\tau,y)\,dyd\tau
\quad\forall\,x\in{\overline D}, \ken m=1,\dots,M,
\end{equation}
with
\begin{equation*}
{\mathcal K}_m(\tau;x,y):=\sum_{\ell=1}^{m}
{\mathfrak X}_{\ssy\Delta_{\ell}}(\tau) \,{\sf G}_{\ssy{\mathcal Q}_{m-\ell+1}}(x,y)
\quad\forall\,\tau\in[0,T],\ken\forall\,x,y\in{\overline D}.
\end{equation*}
%
%
%
Let ${\mathfrak m}\in\{1,\dots,M\}$ and 
${\mathcal E}_{\mathfrak m}:=\left(\,{\mathbb
E}\left[\,\left\|U^{\mathfrak m}-{\widehat u}^{\mathfrak m}
\right\|_{\ssy 0,D}^2\,\right]\,\right)^{\half}$.
Now, we use \eqref{Dinner1}, \eqref{HatUform}, \eqref{WNEQ2}, \eqref{Ito_Isom},
\eqref{HSxar} and \eqref{KatiL2}, to obtain
\begin{equation}\label{EMP_2010_2}
\begin{split}
{\mathcal E}_{\mathfrak m}=&\,\left(\,{\mathbb E}\left[\,\int_{\ssy D}\left(
\int_0^{\ssy T}\!\!\!\int_{\ssy D}
\,\left[\,{\mathcal K}_{\mathfrak m}(\tau;x,y)
-G_{\mathfrak m}(\tau;x,y)\,\right]\,
{\widehat W}(\tau,y)\,dyd\tau\,\right)^2dx\,\right]\,\right)^{\half}\\
\leq&\,\left(\,\int_0^{\tau_{\mathfrak m}}\left(\int_{\ssy D}\!\int_{\ssy D}
\,\left[\,{\mathcal K}_{\mathfrak m}(\tau;x,y)
-G_{\mathfrak m}(\tau;x,y)\,\right]^2\,dydx\right)\,d\tau\,\right)^{\frac{1}{2}}\\
\leq&\,\left(\,\sum_{\ell=1}^{\mathfrak m}\int_{\ssy\Delta_{\ell}}
\left\|{\mathcal Q}_{{\mathfrak m}-\ell+1}
-S(\tau_{\mathfrak m}-\tau)\right\|_{\ssy\rm HS}^2\;d\tau\,\right)^{\half}\\
\leq&\,{\mathcal A}_{\mathfrak m}+{\mathcal B}_{\mathfrak m},
\end{split}
\end{equation}
with
\begin{equation*}
\begin{split}
{\mathcal A}_{\mathfrak m}:=&\,
\left(\,\sum_{\ell=1}^{\mathfrak m}\int_{\ssy\Delta_{\ell}}
\left\|{\mathcal Q}_{{\mathfrak m}-\ell+1}
-S(\tau_{\mathfrak m}-\tau_{\ell-1})\right\|_{\ssy\rm HS}^2\;d\tau\,\right)^{\frac{1}{2}},\\
{\mathcal B}_{\mathfrak m}:=&\,
\left(\,\sum_{\ell=1}^{\mathfrak m}\int_{\ssy\Delta_{\ell}}
\left\|S(\tau_{\mathfrak m}-\tau_{\ell-1})
-S(\tau_{\mathfrak m}-\tau)\right\|_{\ssy\rm HS}^2\;d\tau\,\right)^{\half}.\\
\end{split}
\end{equation*}
Let $\delta\in[0,\frac{3}{4})$. Then, using  the definition of the Hilbert-Schmidt norm,
the deterministic estimate \eqref{September2015_0} and \eqref{SR_BOUND}, we have
\begin{equation*}
\begin{split}
{\mathcal A}_{\mathfrak m}=&\,\left[\,\sum_{\kappa=1}^{\infty}\,\left(\,
\dtau\,\sum_{\ell=1}^{\mathfrak m}
\left\|{\mathcal Q}_{{\mathfrak m}-\ell+1}\varepsilon_{\kappa}
-S(\tau_{\mathfrak m-\ell+1})\varepsilon_{\kappa}\right\|_{\ssy 0,D}^2\,\right)
\,\right]^{\half}\\
=&\,\left[\,\sum_{\kappa=1}^{\infty}\,\left(\,
\dtau\,\sum_{\ell=1}^{\mathfrak m}
\left\|{\mathcal Q}_{\ell}\varepsilon_{\kappa}
-S(\tau_{\ell})\varepsilon_{\kappa}\right\|_{\ssy 0,D}^2\,\right)\,\right]^{\half}\\
\leq&\,C\,\dtau^{\frac{\delta}{2}}\,\left(
\,\sum_{\kappa=1}^{\infty}\|\varepsilon_{\kappa}\|^2_{\ssy
{\bfdot H}^{2\delta-2}}\,\right)^{\half}\\
\leq&\,C\,\dtau^{\frac{\delta}{2}}\,\left(
\,\sum_{\kappa=1}^{\infty}
\tfrac{1}{\lambda_{\kappa}^{4(1-\delta)}}\,\right)^{\half}\\
\leq&\,C\,\dtau^{\frac{\delta}{2}}\,\left(
\,\sum_{\kappa=1}^{\infty}
\tfrac{1}{\lambda_{\kappa}^{1+8(\frac{3}{8}-\frac{\delta}{2})}}\,\right)^{\half}\\
\leq&\,C\,\left(\tfrac{3}{8}-\tfrac{\delta}{2}\right)^{-\half}\,\dtau^{\frac{\delta}{2}},\\
\end{split}
\end{equation*}
which, after setting $\epsilon=\tfrac{3}{8}-\frac{\delta}{2}\in(0,\tfrac{3}{8}]$, yields
\begin{equation}\label{trabajito}
{\mathcal A}_{\mathfrak m}\leq\,C\,\epsilon^{-\frac{1}{2}}\,\dtau^{\frac{3}{8}-\epsilon}.
\end{equation}
Now, using the definition of the Hilbert-Schmidt norm and \eqref{Exx_Sol},
we bound ${\mathcal B}_{\mathfrak m}$ as follows:
\begin{equation*}
\begin{split}
{\mathcal B}_{\mathfrak m}=&\,
\left[\,\sum_{\kappa=1}^{\infty}\,\left(\,\sum_{\ell=1}^{\mathfrak m}\int_{\ssy\Delta_{\ell}}
\left\|S(\tau_{\mathfrak m+1-\ell})\varepsilon_{\kappa}
-S(\tau_{\mathfrak m}-\tau)\varepsilon_{\kappa}\right\|_{\ssy 0,D}^2\;d\tau\right)\,\,\right]^{\half}\\
\leq&\,\left[\,\sum_{\kappa=1}^{\infty}\,\left(\,
\sum_{\ell=1}^{\mathfrak m}\int_{\ssy\Delta_{\ell}} \left(\int_{\ssy D}\left[
e^{-\lambda_{\kappa}^4(\tau_{\mathfrak m}-\tau_{\ell-1})}
-e^{-\lambda_{\kappa}^4(\tau_{\mathfrak m}-\tau)}\right]^2
\varepsilon_{\kappa}^2(x)\,dx\right)
\,d\tau\,\right)\,\right]^{\half}\\
\leq&\,\left[\,\sum_{\kappa=1}^{\infty}\left(\,
\sum_{\ell=1}^{\mathfrak m}\int_{\ssy\Delta_{\ell}}
e^{-2\lambda_{\kappa}^4(\tau_{\mathfrak m}-\tau)} \left(\,1-
e^{-\lambda_{\kappa}^4(\tau-\tau_{\ell-1})}\,\right)^2
\,d\tau\,\right)\,\right]^{\half}\\
\leq&\,\left[\,\sum_{\kappa=1}^{\infty} \left(\,1-e^{-\lambda_{\kappa}^4\,\dtau}\,\right)^2
\left(\, \int_0^{\tau_{\mathfrak m}} e^{-2\lambda_{\kappa}^4(\tau_{\mathfrak m}-\tau)}
\,d\tau\,\right)\,\right]^{\half}\\
\leq&\,\tfrac{1}{\sqrt{2}}\,\left(\,\sum_{\kappa=1}^{\infty}
\tfrac{1-e^{-2\lambda_{\kappa}^4\,\dtau}}{\lambda_{\kappa}^4}\,\right)^{\half},
\end{split}
\end{equation*}
from which, applying (3.13) in \cite{KZ2010}, we obtain
\begin{equation}\label{Ydaspis952}
{\mathcal B}_{\mathfrak m}\leq\,C\,
\,\,\Delta\tau^{\frac{3}{8}}.
\end{equation}
\par
Finally, the estimate \eqref{Dinner0} follows easily combining \eqref{EMP_2010_2},
\eqref{trabajito} and \eqref{Ydaspis952}.
\end{proof}
%
%
%
%
%
\section{Fully-Discrete Approximations}\label{SECTION44}
\subsection{The deterministic problem}\label{SECTION44a}
In this section we construct and analyze finite element approximations, $(W_h^m)_{m=0}^{\ssy M}$,
of the modified Crank-Nicolson time-discrete approximations defined in Section~\ref{section3A}.
%
%
\par
Let $p=2$ or $3$. We begin, by setting
\begin{equation}\label{CNFD1}
W_h^0:=P_hw_0
\end{equation}
and then by finding $W_h^1\in{\sf S}_h^p$ such that
\begin{equation}\label{CNFD12}
W_h^1-W_h^0+\tfrac{\dtau}{2}\,B_hW_h^1 =0.
\end{equation}
Finally, for $m=2,\dots,M$, we specify $W_h^m\in{\sf S}_h^p$ such
that
\begin{equation}\label{CNFD2}
W_h^m-W_h^{m-1}+\dtau\,B_hW_h^{m-\frac{1}{2}}=0.
\end{equation}
\par
First,  we show a discrete in time $L^2_t(L^2_x)$ a priori estimate of time averages
of the nodal error between the modified Crank-Nicolson time-discrete approximations and
the modified Crank-Nicolson fully-discrete approximations defined above.
%
%
\begin{proposition}\label{Aygo_Kokora}
Let $p=2$ or $3$, $w$ be the solution of the problem \eqref{Det_Parab},
$(W^m)_{m=0}^{\ssy M}$ be the Crank-Nicolson time-discrete
approxi\-mations of $w$ defined by \eqref{CNDet1}--\eqref{CNDet2},
and $(W_h^m)_{m=0}^{\ssy M}$ be the modified Crank-Nicolson
fully-discrete approximations of $w$ specified by \eqref{CNFD1}--\eqref{CNFD2}.
Then, there exists a constant $C>0$,
independent of $h$ and $\dtau$, such that
\begin{equation}\label{tiger_river1}
\left(\,\dtau\,\|W^{1}-W_h^{1}\|^2_{\ssy 0,D}
+\dtau\sum_{m=2}^{\ssy M}
\|W^{m-\frac{1}{2}}-W_h^{m-\frac{1}{2}}\|^2_{\ssy
0,D}\,\right)^{\frac{1}{2}}\leq \,C\,\,h^{p\,\theta}
\,\,\|w_0\|_{\ssy {\bfdot H}^{3\theta-2}}
\end{equation}
for all $\theta\in[0,1]$ and $w_0\in {\bfdot H}^2(D)$.
%
%
%
%
%
%
\end{proposition}
%
%
%
%
%
%
%
\begin{proof}
We will get the error estimate \eqref{tiger_river1} by interpolation after proving it for $\theta=1$ and 
$\theta=0$ (cf. \cite{KZ2010}). In the sequel,  we will use the symbol
$C$ to denote a generic constant that is independent of $\dtau$ and may changes value from
one line to the other.
\par\noindent
\vskip0.3truecm
\par
\textbullet\, {\sf Case} $\theta=1$:
Letting ${\sf \Theta}^{\ell}:=W^{\ell}-W_h^{\ell}$ for $\ell=0,\dots,M$, we use
\eqref{CNDet12}, \eqref{CNFD12}, \eqref{CNDet2} and \eqref{CNFD2},
to arrive at the following error equations:
\begin{equation}\label{SaintMinas_1a}
T_{\ssy B,h}({\sf\Theta}^1-{\sf\Theta}^{0})+\tfrac{\dtau}{2}
\,{\sf\Theta}^{1}=\tfrac{\dtau}{2}\,\xi_1
\end{equation}
and
\begin{equation}\label{SaintMinas_1}
T_{\ssy B,h}({\sf\Theta}^m-{\sf\Theta}^{m-1})+\Delta\tau\,{\sf\Theta}^{m-\frac{1}{2}}
=\dtau\,\xi_m,\quad m=2,\dots,M,
\end{equation}
where
\begin{equation}\label{xi_def_1}
\xi_1:=(T_{\ssy B}-T_{\ssy B,h})\partial_x^4W^{1}
\end{equation}
and
\begin{equation}\label{xi_def_2}
\xi_{\ell}:=(T_{\ssy B}-T_{\ssy B,h})\partial_x^4W^{\ell-\frac{1}{2}},
\quad\ell=2,\dots,M.
\end{equation}
Taking the $L^2(D)-$inner product of both sides of
\eqref{SaintMinas_1} with ${\sf \Theta}^{m-\frac{1}{2}}$ and then using \eqref{adjo2},
the Cauchy-Schwarz inequality along with the arithmetic mean inequality, we obtain
\begin{equation*}
\|\partial_x^2(T_{\ssy B,h}{\sf\Theta}^m)\|_{\ssy 0,D}^2
-\|\partial_x^2(T_{\ssy B,h}{\sf\Theta}^{m-1})\|_{\ssy 0,D}^2
+\dtau\,\|{\sf\Theta}^{m-\frac{1}{2}}\|_{\ssy
0,D}^2\leq{\Delta\tau}\,\|\xi_m\|_{\ssy 0,D}^2,\quad m=2,\dots,M.
\end{equation*}
After summing with respect to $m$ from $2$ up to $M$, the relation above yields
\begin{equation*}
\dtau\,\sum_{m=2}^{\ssy M}\|{\sf\Theta}^{m-\frac{1}{2}}\|_{\ssy 0,D}^2
\leq\,\dtau\sum_{m=2}^{\ssy M}\left\|\xi_m\right\|_{\ssy 0,D}^2
+\|\partial_x^2(T_{\ssy B,h}{\sf\Theta}^{1})\|_{\ssy 0,D}^2,
\end{equation*}
which, easily, yields that
\begin{equation}\label{citah3}
\dtau\,\|{\sf\Theta}^1\|_{\ssy 0,D}^2
+\dtau\,\sum_{m=2}^{\ssy M}\|{\sf\Theta}^{m-\frac{1}{2}}\|_{\ssy 0,D}^2
\leq\,\|\partial_x^2(T_{\ssy B,h}{\sf\Theta}^{1})\|_{\ssy 0,D}^2
+\dtau\,\|{\sf\Theta}^1\|^2_{\ssy 0,D}
+\dtau\sum_{m=2}^{\ssy M}\left\|\xi_m\right\|_{\ssy 0,D}^2.
\end{equation}
Observing that $T_{\ssy B,h}{\sf\Theta}^0=0$, we
take the $L^2(D)-$inner product of both sides of
\eqref{SaintMinas_1a} with ${\sf \Theta}^{1}$ and then 
use the Cauchy-Schwarz inequality  along with the arithmetic mean inequality, to get
\begin{equation}\label{July2015_1}
\|\partial_x^2(T_{\ssy B,h}{\sf\Theta^1})\|_{\ssy 0,D}^2+\tfrac{\dtau}{4}\,
\|{\sf\Theta^1}\|^2_{\ssy 0,D}\leq\,\tfrac{\dtau}{4}\,\|\xi_1\|_{\ssy 0,D}^2.
\end{equation}
Thus, using \eqref{citah3}, \eqref{July2015_1},
\eqref{xi_def_1}, \eqref{xi_def_2} and \eqref{ARA1},
we easily conclude that
\begin{equation}\label{July2015_2}
\begin{split}
\dtau\,\|{\sf\Theta}^1\|_{\ssy 0,D}^2
+\dtau\,\sum_{m=2}^{\ssy M}\|{\sf\Theta}^{m-\frac{1}{2}}\|_{\ssy 0,D}^2
\leq&\,\dtau\,\|\xi_1\|_{\ssy 0,D}^2
+\dtau\,\sum_{m=2}^{\ssy M}\left\|\xi_m\right\|_{\ssy 0,D}^2\\
\leq&\,C\,h^{2p}\,\left(\,\dtau\,\|\partial^3_xW^1\|^2_{\ssy 0,D}
+\dtau\,\sum_{m=2}^{\ssy M} 
\|\partial_x^3W^{m-\frac{1}{2}}\|_{\ssy 0,D}^2\right).
\end{split}
\end{equation}
%
%
%
Taking the $L^2(D)-$inner product of \eqref{CNDet2} with
$\partial_x^2W^{m-\frac{1}{2}}$,
and then integrating by parts and summing with respect to $m$,
from $2$ up to $M$, it follows that
\begin{equation*}
\|\partial_xW^{\ssy M}\|_{\ssy 0,D}^2-\|\partial_xW^1\|_{\ssy 0,D}^2
+2\,\dtau\,\sum_{m=2}^{\ssy M}
\|\partial_x^3W^{m-\frac{1}{2}}\|_{\ssy 0,D}^2=0
\end{equation*}
which yields
\begin{equation}\label{citah6}
\dtau\,\|\partial_x^3W^1\|_{\ssy 0,D}^2
+\sum_{m=2}^{\ssy M}
\dtau\,\|\partial_x^3W^{m-\frac{1}{2}}\|_{\ssy 0,D}^2
\leq\,\tfrac{1}{2}\,\|\partial_xW^1\|_{\ssy 0,D}^2
+\dtau\,\|\partial_x^3W^1\|_{\ssy 0,D}^2.
\end{equation}
Now, take the $L^2(D)-$inner product of \eqref{CNDet12}
with $\partial_x^2W^1$, and then
integrate by parts and use \eqref{innerproduct} to get
\begin{equation*}
\|\partial_xW^1\|^2_{\ssy 0,D}-\|\partial_xW^0\|_{\ssy 0,D}^2
+\dtau\,\|\partial_x^3W^1\|_{\ssy 0,D}^2\leq0,
\end{equation*}
which, along with \eqref{H_equiv}, yields
\begin{equation}\label{citah6a}
\|\partial_xW^1\|_{\ssy 0,D}^2
+\dtau\,\|\partial_x^3W^1\|_{\ssy 0,D}^2\leq\,\|w_0\|_{\ssy {\bfdot H}^1}^2.
\end{equation}
Thus, combining \eqref{July2015_2}, \eqref{citah6} and \eqref{citah6a},
we obtain \eqref{tiger_river1} for $\theta=1$.
%
%
%
\par\noindent
\vskip0.3truecm
\par
\textbullet\, {\sf Case} $\theta=0$:
From \eqref{CNFD12} and \eqref{CNFD2}, it follows that
\begin{equation}\label{Frago_0}
T_{\ssy B,h}(W_h^1-W_h^0)+\tfrac{\dtau}{2}\,W_h^1=0
\end{equation}
and
\begin{equation}\label{Frago_1}
T_{\ssy B,h}(W_h^m-W_h^{m-1})+\dtau\,W_h^{m-\frac{1}{2}}=0,
\quad m=2,\dots,M.
\end{equation}
Taking the $L^2(D)-$inner product
of \eqref{Frago_1} with $W_h^{m-\frac{1}{2}}$ and using \eqref{adjo2},
we have
\begin{equation*}
\|\partial_x^2(T_{\ssy B,h}W_h^m)\|_{\ssy 0,D}^2
-\|\partial_x^2(T_{\ssy B,h}W_h^{m-1})\|_{\ssy 0,D}^2
+2\,\dtau\,\|W_h^{m-\frac{1}{2}}\|_{\ssy 0,D}^2=0, \quad
m=2,\dots,M,
\end{equation*}
which, after summing with respect to $m$ from $2$ up to $M$, yields
\begin{equation}\label{Frago_2}
\dtau\,\|W_h^1\|^2_{\ssy 0,D}
+\sum_{m=2}^{\ssy M} \Delta\tau\,\|W_h^{m-\frac{1}{2}}\|_{\ssy
0,D}^2 \leq\tfrac{1}{2}\,\|\partial_x^2(T_{\ssy B,h}W_h^1)\|_{\ssy
0,D}^2+\dtau\,\|W_h^1\|^2_{\ssy 0,D}.
\end{equation}
Now, take the $L^2(D)-$inner product
of \eqref{Frago_0} with $W_h^1$ and use \eqref{innerproduct}
and \eqref{CNFD1}, to have
\begin{equation}\label{Frago_3}
\begin{split}
\|\partial_x^2(T_{\ssy B,h}W_h^1)\|_{\ssy 0,D}^2
+\dtau\,\|W_h^{1}\|_{\ssy 0,D}^2\leq&\,
\|\partial_x^2(T_{\ssy B,h}P_hw_0)\|_{\ssy 0,D}^2\\
\leq&\,\|\partial_x^2(T_{\ssy B,h}w_0)\|_{\ssy 0,D}^2.\\
\end{split}
\end{equation}
Combining \eqref{Frago_2}, \eqref{Frago_3}, \eqref{TB_bound}
and \eqref{minus_equiv}, we obtain
\begin{equation}\label{Frago_4}
\left(\dtau\,\|W_h^1\|^2_{\ssy 0,D}
+\sum_{m=2}^{\ssy M}\dtau
\,\|W_h^{m-\frac{1}{2}}\|_{\ssy 0,D}^2\right)^{\half}
\leq\,C\,\|w_0\|_{\ssy{\bfdot H}^{-2}}.
\end{equation}
Finally, combine \eqref{Frago_4} with
\eqref{May2015_23} to get \eqref{tiger_river1} for $\theta=0$.
%
%
\end{proof}
%
%
\par
Next,  we derive a discrete in time $L^2_t(L^2_x)$ a priori estimate
of the nodal error between the modified Crank-Nicolson time-discrete approximations and
the modified Crank-Nicolson fully-discrete approximations.
%
%
\begin{proposition}\label{X_Aygo_Kokora}
Let $p=2$ or $3$, $w$ be the solution of the problem \eqref{Det_Parab},
$(W^m)_{m=0}^{\ssy M}$ be the modified Crank-Nicolson time-discrete
approximations of $w$ defined by \eqref{CNDet1}--\eqref{CNDet2},
and $(W_h^m)_{m=0}^{\ssy M}$ be the modified Crank-Nicolson finite element
approximations of $w$ specified by \eqref{CNFD1}--\eqref{CNFD2}.
Then, there exists a constant $C>0$,
independent of $h$ and $\dtau$, such that
\begin{equation}\label{X_tiger_river1}
\left(\,\dtau\sum_{m=1}^{\ssy M}
\|W^{m}-W_h^{m}\|^2_{\ssy
0,D}\,\right)^{\frac{1}{2}}\leq \,C\,\left[\,\dtau^{\frac{\delta}{2}}
\,\|w_0\|_{\ssy {\bfdot H}^{2(\delta-1)}}
+h^{p\,\theta}
\,\,\|w_0\|_{\ssy {\bfdot H}^{3\theta-2}}\right]
\end{equation}
for all $\delta$, $\theta\in[0,1]$ and $w_0\in {\bfdot H}^2(D)$.
\end{proposition}
%
%
%
%
%
\begin{proof}
The proof is based on the estimation of the terms
in the right hand side of the following triangle inequality:
\begin{equation}\label{X_Hadji0}
\left(\dtau\sum_{m=1}^{\ssy M}\|W^m-W_h^m\|_{\ssy 0,D}^2\right)^{\frac{1}{2}}\leq
{\mathcal S}_{\ssy B}+{\mathcal S}_{\ssy C}+{\mathcal S}_{\ssy D}
\end{equation}
where
\begin{equation*}
\begin{split}
{\mathcal S}_{\ssy B}:=&\,\left(\dtau\sum_{m=2}^{\ssy M}\|W^m-W^{m-\half}\|_{\ssy 0,D}^2\right)^{\frac{1}{2}},\\
{\mathcal S}_{\ssy C}:=&\,\left(\dtau\,\|W^1-W_h^1\|_{\ssy 0,D}^2
+\dtau\sum_{m=2}^{\ssy M}\|W^{m-\half}-W_h^{m-\half}\|_{\ssy 0,D}^2\right)^{\frac{1}{2}},\\
{\mathcal S}_{\ssy D}:=&\,\left(\dtau\sum_{m=2}^{\ssy M}\|W_h^{m-\half}-W_h^m\|_{\ssy 0,D}^2\right)^{\frac{1}{2}}.\\
\end{split}
\end{equation*}
In the sequel,  we will use the symbol $C$ to denote a generic constant that is
independent of $\dtau$ and may changes value from one line to the other.
\par
Taking the $L^2(D)-$inner product of both sides of \eqref{CNFD2} with $(W_h^m-W_h^{m-1})$,
we have
\begin{equation}\label{X_Hadji1}
\|W_h^m-W_h^{m-1}\|^2_{\ssy 0,D}+\tfrac{\dtau}{2}\,\left(\,\|\partial_x^2W_h^m\|_{\ssy 0,D}^2
-\|\partial_x^2W_h^{m-1}\|_{\ssy 0,D}^2\,\right)=0,\quad m=2,\dots,M.
\end{equation}
After summing both sides of \eqref{X_Hadji1} with respect to $m$ from $2$ up to $M$, we obtain
\begin{equation*}
\dtau\sum_{m=2}^{\ssy M}\|W_h^m-W_h^{m-1}\|^2_{\ssy 0,D}
+\tfrac{\dtau^2}{2}\,\left(\,\|\partial_x^2W_h^{\ssy M}\|_{\ssy 0,D}^2
-\|\partial_x^2W_h^{1}\|_{\ssy 0,D}^2\,\right)=0,
\end{equation*}
which yields
\begin{equation}\label{X_Hadji2}
\dtau\sum_{m=2}^{\ssy M}\|W_h^m-W_h^{m-1}\|^2_{\ssy 0,D}
\leq\tfrac{\dtau^2}{2}\,\|\partial_x^2W_h^{1}\|_{\ssy 0,D}^2.
\end{equation}
Taking the $L^2(D)-$inner product of both sides of \eqref{CNFD12} with $W_h^1$
and then using \eqref{innerproduct}, we obtain
\begin{equation*}
\|W_h^1\|_{\ssy 0,D}^2-\|W_h^0\|_{\ssy 0,D}^2+\dtau\,\|\partial_x^2W_h^1\|_{\ssy 0,D}^2\leq0
\end{equation*}
from which we conclude that
\begin{equation}\label{X_Hadji3}
\dtau\,\|\partial_x^2W_h^1\|_{\ssy 0,D}^2\leq\|w_0\|_{\ssy 0,D}^2.
\end{equation}
Thus, combining \eqref{X_Hadji2} and \eqref{X_Hadji3} we have
\begin{equation}\label{X_Hadji4}
\begin{split}
{\mathcal S}_{\ssy D}=&\,\tfrac{1}{2}\,
\left(\dtau\sum_{m=2}^{\ssy M}\|W_h^m-W_h^{m-1}\|_{\ssy 0,D}^2\right)^{\half}\\
\leq&\,\tfrac{1}{2\sqrt{2}}\,\dtau\,\|\partial_x^2W^1_h\|_{\ssy 0,D}\\
\leq&\,\dtau^{\half}\,\|w_0\|_{\ssy 0,D}.\\
\end{split}
\end{equation}
Taking again the $L^2(D)-$inner product of both sides of \eqref{CNFD12} with $W_h^1$
and then using \eqref{minus_equiv} and \eqref{H_equiv} along with the
arithmetic mean inequality,  we obtain
\begin{equation}\label{X_Hadji5}
\begin{split}
\|W_h^1\|_{\ssy 0,D}^2+\tfrac{\dtau}{2}
\,\|\partial_x^2W_h^1\|_{\ssy 0,D}^2=&\,(P_hw_0,W_h^1)_{\ssy 0,D}\\
=&\,(w_0,W_h^1)_{\ssy 0,D}\\
\leq&\,\|w_0\|_{\ssy -2,D}\,\|W_h^1\|_{\ssy 2,D}\\
\leq&\,C\,\|w_0\|_{\ssy{\bfdot H}^{-2}}\,\|W_h^1\|_{\ssy{\bfdot H}^2}\\
\leq&\,C\,\dtau^{-1}\|w_0\|^2_{\ssy{\bfdot H}^{-2}}
+\tfrac{\dtau}{4}\,\|W_h^1\|_{\ssy{\bfdot H}^2}^2.
\end{split}
\end{equation}
Observing that $\|\partial_x^2W_h^1\|_{\ssy 0,D}=\|W_h^1\|_{\ssy {\bfdot H}^2}$
(cf. \eqref{X_evid}), \eqref{X_Hadji5} yields that
\begin{equation}\label{X_Hadji6}
\dtau^2\,\|\partial_x^2W_h^1\|_{\ssy 0,D}^2\leq\,C\,\|w_0\|_{\ssy{{\bfdot H}^{-2}}}^2.
\end{equation}
Thus, combining \eqref{X_Hadji2} and \eqref{X_Hadji6}, we conclude that
\begin{equation}\label{X_Hadji8}
\begin{split}
{\mathcal S}_{\ssy D}
=&\,\tfrac{1}{2}\,
\left(\dtau\sum_{m=2}^{\ssy M}\|W_h^m-W_h^{m-1}\|_{\ssy 0,D}^2\right)^{\half}\\
\leq&\,\tfrac{1}{2\sqrt{2}}\,\dtau\,\|\partial_x^2W^1_h\|_{\ssy 0,D}\\
\leq&\,C\,\|w_0\|_{\ssy{{\bfdot H}^{-2}}}.
\end{split}
\end{equation}
Also, from \eqref{Hadji4} and \eqref{Hadji8}, we have
\begin{equation}\label{XX_Hadji4}
{\mathcal S}_{\ssy B}\leq\,C\,\dtau^{\half}\,\|w_0\|_{\ssy 0,D}
\end{equation}
and
\begin{equation}\label{XX_Hadji44}
{\mathcal S}_{\ssy B}\leq\,C\,\|w_0\|_{\ssy{{\bfdot H}^{-2}}}.
\end{equation}
By interpolation, from \eqref{X_Hadji4}, \eqref{XX_Hadji4}, \eqref{X_Hadji8}
and \eqref{XX_Hadji44}, we conclude that
\begin{equation}\label{XXX_1}
{\mathcal S}_{\ssy B}+{\mathcal S}_{\ssy D}\leq\,C\,\dtau^{\frac{\delta}{2}}
\,\|w_0\|_{\ssy{{\bfdot H}^{2(\delta-1)}}}
\quad\forall\,\delta\in[0,1].
\end{equation}
Finally, the estimate \eqref{tiger_river1} reads
\begin{equation}\label{XXX_2}
{\mathcal S}_{\ssy C}\leq\,C\,h^{p\,\theta}\,\|w_0\|_{\ssy{\bfdot H}^{3\theta-2}}
\quad\forall\,\theta\in[0,1].
\end{equation}
Thus, \eqref{X_tiger_river1} follows as a simple consequence of
\eqref{X_Hadji0}, \eqref{XXX_1} and \eqref{XXX_2}.
\end{proof}
%
%
\subsection{The stochastic problem}
The following lemma ensures the existence of a continuous
Green function for some discrete operators (cf. Lemma~5.2 in \cite{KZ2010},
Lemma~5.1 in \cite{KZ2013b}).
%
%
\begin{lemma}\label{prasinolhmma}
Let $p=2$ or $3$, $f\in L^2(D)$ and $g_h$, $\psi_h$,  $z_h\in{\sf S}_h^p$ such that
\begin{equation}\label{AuxEq_1}
\psi_h+\tfrac{\dtau}{2}\,B_h\psi_h=P_hf
\end{equation}
and
\begin{equation}\label{AuxEq_2}
z_h=g_h-\tfrac{\dtau}{2}\,B_hg_h.
\end{equation}
Then there exist functions $G_h$, ${\widetilde G}_h\in C({\overline {D\times D}})$
such that
\begin{equation}\label{prasinogreen1}
\psi_h(x)=\int_{\ssy D} G_h(x,y)\,f(y)\,dy\quad\forall\,x\in{\overline D}
\end{equation}
and
\begin{equation}\label{prasinogreen2}
z_h(x)=\int_{\ssy D}{\widetilde G}_h(x,y)\,g_h(y)\,dy\quad\forall\,x\in{\overline D}.
\end{equation}
\end{lemma}
%
%
%
%
\begin{proof}
Let $\nu_h:=\text{\rm dim}({\sf S}_h^p)$ and $\gamma: {\sf S}_h^p\times{\sf S}_h^p\rightarrow\rset$
be an inner product on ${\sf S}_h^p$ defined by
$\gamma(\chi, \varphi) :={\mathcal B}(\chi,\varphi)_{\ssy 0,D}$
for $\chi,\varphi\in {\sf S}_h^p$.
Then, we can construct a basis $(\varphi_j)_{j=1}^{\ssy \nu_h}$ of ${\sf S}_h^p$
which is $L^2(D)-$orthonormal, i.e.,
$(\varphi_i,\varphi_j)_{\ssy 0,D}=\delta_{ij}$ for 
$i,j=1,\dots,\nu_h$,
and $\gamma-$orthogonal, i.e. there exist positive
$(\varepsilon_{h,j})_{j=1}^{\ssy \nu_h}$ such that
$\gamma(\varphi_i,\varphi_j)=\varepsilon_{h,i}\,\delta_{ij}$
for $i,j=1,\dots,\nu_h$ 
(see Sect. 8.7 in \cite{Golub}). Thus, there exist real numbers $(\mu_j)_{j=1}^{\nu_h}$
and $({\widetilde\mu}_j)_{j=1}^{\nu_h}$ such that $\psi_h=\sum_{j=1}^{\nu_h}\mu_j\,\varphi_j$
and $z_h=\sum_{j=1}^{\nu_h}{\widetilde\mu}_j\,\varphi_j$. Then, \eqref{AuxEq_1} and
\eqref{AuxEq_2} yield
\begin{equation}\label{Aram_1}
\mu_{\ell}=\tfrac{1}{1+\frac{\dtau}{2}\,\varepsilon_{h,\ell}}\,\,(f,\varphi_{\ell})_{\ssy 0,D},
\quad
{\widetilde\mu}_{\ell}=\left(1-\tfrac{\dtau}{2}\,\varepsilon_{h,\ell}\right)\,(g_h,\varphi_{\ell})_{\ssy 0,D}
\end{equation}
for $\ell=1,\dots,\nu_h$. Using \eqref{Aram_1} we conclude \eqref{prasinogreen1} and 
\eqref{prasinogreen2} with
$G_h(x,y)=\sum_{j=1}^{\nu_h}
\tfrac{1}{1+\frac{\dtau}{2}\,\varepsilon_{h,j}}\,\varphi_j(x)\,\varphi_j(y)$
and 
${\widetilde G}_h(x,y)=\sum_{j=1}^{\nu_h}
(1-\tfrac{\dtau}{2}\,\varepsilon_{h,j})\,\varphi_j(x)\,\varphi_j(y)$
for $x$, $y\in{\overline D}$.
\end{proof}
%
%
%
%
Now, we are ready to derive a discrete in time $L^{\infty}_t(L^2_{\ssy P}(L^2_x))$ a priori estimate
of the nodal error between the Crank-Nicolson time-discrete approximations of ${\widehat u}$ and
the Crank-Nicolson fully-discrete approximations of ${\widehat u}$.
%
%
%
%
\begin{theorem}\label{Tigrakis}
Let $p=2$ or $3$, ${\widehat u}$ be the solution of the problem
\eqref{AC2}, $(U_h^m)_{m=0}^{\ssy M}$ be the Crank-Nicolson
fully-discrete approximations of ${\widehat u}$ defined by \eqref{FullDE1} and
\eqref{FullDE2},
and $(U^m)_{m=0}^{\ssy M}$ be the Crank-Nicolson time-discrete
approximations of ${\widehat u}$ defined by \eqref{CN_S1} and \eqref{CN_S2}.
Then, there exists a constant $C>0$, independent of ${\sf M}_{\star}$,
${\Delta t}$, $h$ and $\dtau$,  such that
\begin{equation}\label{Lasso1}
\max_{1\leq{m}\leq {\ssy M}}\left({\mathbb E}\left[
\big\|U_h^m-U^m\big\|^2_{\ssy 0,D}\right]\right)^{\half}
\leq\,C\,\left(\,\epsilon_1^{-\half}\,\,\,h^{\frac{p}{2}-\epsilon_1}
+\epsilon_2^{-\frac{1}{2}}\,\dtau^{\frac{3}{8}-\epsilon_2}\right)
\quad\forall\,\epsilon_1\in\left(0,\tfrac{p}{2}\right],
\quad\forall\,\epsilon_2\in\left(0,\tfrac{3}{8}\right].
\end{equation}
%
\end{theorem}
%
%
%
%
%
%
%
%
%
\begin{proof}
Let ${\sf I}:L^2(D)\to L^2(D)$ be the identity operator,
${\sf Y}_h:{\sf S}_h^p\rightarrow{\sf S}_h^p$ be defined
by ${\sf Y}_h:={\sf I}-\tfrac{\dtau}{2}\,B_h$
and $\Lambda_h:L^2(D)\to{\sf S}^p_h$ be the inverse discrete elliptic
operator given by $\Lambda_h:=({\sf I}+\tfrac{\dtau}{2}\,B_h)^{-1}P_h$.
Also, for $m=1,\dots,M$, we define a discrete operator
${\mathcal Q}_{h,m}:L^2(D)\rightarrow{\sf S}_h^p$ by
${\mathcal Q}_{h,m}:=(\Lambda_h{\sf Y}_h)^{m-1}\Lambda_h$,
which has a Green function ${\sf G}_{{\mathcal Q}_{h,m}}$ (cf.
Lemma~\ref{prasinolhmma}).
%
%
Using, now, an induction argument, from \eqref{FullDE2} we
conclude that
\begin{equation*}
U_h^m=\sum_{\ell=1}^{\ssy m} \int_{\ssy\Delta_{\ell}}
{\mathcal Q}_{h,m-\ell+1}({\widehat W}(\tau,\cdot))\,d\tau,
\quad m=1,\dots,M,
\end{equation*}
which is written, equivalently, as follows:
\begin{equation}\label{Anaparastash2}
U_h^m(x)=\int_0^{\ssy T}\!\!\!\int_{\ssy D}
\,{\mathcal K}_{h,m}(\tau;x,y)\,{\widehat W}(\tau,y)
\,dyd\tau\quad\forall\,x\in{\overline D}, \ken m=1,\dots,M,
\end{equation}
where
\begin{equation*}
{\mathcal K}_{h,m}(\tau;x,y)
:=\sum_{\ell=1}^m{\mathfrak X}_{\ssy\Delta_{\ell}}(\tau)
\,{\sf G}_{{\mathcal Q}_{h,m-\ell+1}}(x,y)\quad\forall
\,\tau\in[0,T],\ken\forall\,x,y\in{\overline D}.
\end{equation*}
%
%
%
Using \eqref{Anaparastash2}, \eqref{Dinner1}, \eqref{WNEQ2}, \eqref{Ito_Isom},
\eqref{HSxar} and \eqref{KatiL2}, we get
\begin{equation*}
\begin{split}
\left(\,{\mathbb E}\left[\,\|U^{m}-U_h^{m}\|_{\ssy 0,D}^2\,\right]\,\right)^{\frac{1}{2}}
\leq&\,\left(\,\int_0^{\tau_m}\left(\int_{\ssy D}\!\int_{\ssy D}
\,\left[{\mathcal K}_{m}(\tau;x,y)
-{\mathcal K}_{h,m}(\tau;x,y)\right]^2
\,dydx\right)\,d\tau\,\right)^{\frac{1}{2}}\\
\leq&\,\left(\,\dtau\sum_{\ell=1}^{m}
\left\|{\mathcal Q}_{m-\ell+1}
-{\mathcal Q}_{h,m-\ell+1}\right\|_{\ssy\rm HS}^2\right)^{\frac{1}{2}},
\quad m=1,\dots,M.\\
\end{split}
\end{equation*}
%
%
Let $\delta\in\left[0,\tfrac{3}{4}\right)$ and $\theta\in\left[0,\frac{1}{2}\right)$.
Using the definition of the Hilbert-Schmidt norm and
the deterministic error estimate \eqref{X_tiger_river1},
we obtain
\begin{equation*}
\begin{split}
%
%
\left(\,{\mathbb E}\left[\,\|U^{m}-U_h^{m}\|_{\ssy 0,D}^2\,
\right]\,\right)^{\frac{1}{2}}\leq&\,\left(\,\sum_{k=1}^{\infty}
\,\left[\,\dtau\sum_{\ell=1}^{m}
\left\|{\mathcal Q}_{m-\ell+1}\varepsilon_k-{\mathcal Q}_{h,m-\ell+1}\varepsilon_k
\right\|_{\ssy 0,D}^2\right]\,\right)^{\half}\\
\leq&\,\left(\,\sum_{k=1}^{\infty}
\,\left[\,\dtau\,\sum_{\ell=1}^{m}
\left\|{\mathcal Q}_{\ell}\varepsilon_k-{\mathcal Q}_{h,\ell}\varepsilon_k
\right\|_{\ssy 0,D}^2\right]\,\right)^{\half}\\
%
\leq&\,C\,\left[h^{2\,p\,\theta}
\,\sum_{k=1}^{\infty}
\|\varepsilon_k\|^2_{\ssy {\bfdot H}^{3\theta-2}}
+\dtau^{\delta}\,\sum_{k=1}^{\infty}
\|\varepsilon_k\|^2_{\ssy {\bfdot H}^{2(\delta-1)}}\right]^{\frac{1}{2}}\\
\leq&\,C\,h^{p\,\theta}\,\left(
\,\sum_{k=1}^{\infty} \tfrac{1}{\lambda_k^{2\,(2-3\theta)}}\,\right)^{\half}
+\dtau^{\frac{\delta}{2}}\,\left(\,
\sum_{k=1}^{\infty} \tfrac{1}{\lambda_k^{4(1-\delta)}}\,\right)^{\half}\\
\leq&\,C\,h^{p\,\theta}\,\left(
\,\sum_{k=1}^{\infty} \tfrac{1}{\lambda_k^{1+6\,(\frac{1}{2}-\theta)}}\,\right)^{\half}
+\dtau^{\frac{\delta}{2}}\,\left(\,
\sum_{k=1}^{\infty} \tfrac{1}{\lambda_k^{1+8\,(\frac{3}{8}-\frac{\delta}{2})}}\,\right)^{\half},
\quad m=1,\dots,M,
\end{split}
\end{equation*}
which, along with \eqref{SR_BOUND}, yields
\begin{equation}\label{Dec2015_10th}
\max_{1\leq{m}\leq{\ssy M}}\left(\,{\mathbb E}\left[\,\|U^{m}-U_h^{m}\|_{\ssy 0,D}^2\,
\right]\,\right)^{\frac{1}{2}}\leq\,C\,\left[\,h^{p\,\theta}\,\left(\tfrac{1}{2}-\theta\right)^{-\half}
+\dtau^{\frac{\delta}{2}}\,\left(\tfrac{3}{8}-\tfrac{\delta}{2}\right)^{-\half}\,\right].
\end{equation}
Setting $\epsilon_1=p\,\left(\frac{1}{2}-\theta\right)\in\left(0,\frac{p}{2}\right]$ and
$\epsilon_2=\frac{3}{8}-\tfrac{\delta}{2}\in\left(0,\tfrac{3}{8}\right]$,
we, easily, arrive at \eqref{Lasso1}.
\end{proof}
%
%
%
%
%
%
%
Now, we are able to formulate a discrete in time $L^{\infty}_t(L^2_{\ssy P}(L^2_x))$
error estimate for the Crank-Nicolson fully-discrete approximations of ${\widehat u}$.
%
%
%
%
\begin{theorem}\label{FFQEWR}
Let $p=2$ or $3$, ${\widehat u}$ be the solution of problem \eqref{AC2},
and $(U_h^m)_{m=0}^{\ssy M}$ be the Crank-Nicolson
fully-discrete approximations of ${\widehat u}$ constructed by
\eqref{FullDE1}-\eqref{FullDE2}. Then, there exists a constant
$C>0$, independent of  ${\sf M}_{\star}$, ${\Delta t}$, $h$ and $\dtau$,
such that
\begin{equation}\label{Final_Estimate}
\max_{0\leq{m}\leq{\ssy M}}\left({\mathbb E}\left[
 \|U_h^m-{\widehat u}(\tau_m,\cdot)\|_{\ssy 0,D}^2\right]
\right)^{\half} \leq\,C\,\left(\,
\epsilon_1^{-\frac{1}{2}}\,h^{\frac{p}{2}-\epsilon_1}
+\epsilon_2^{-\frac{1}{2}}\,\dtau^{\frac{3}{8}-\epsilon_2}\,\right)
\quad\forall\,\epsilon_1\in\left(0,\tfrac{p}{2}\right],
\quad\forall\,\epsilon_2\in\left(0,\tfrac{3}{8}\right].
\end{equation}
%
%
\end{theorem}
%
%
%
%
%
%
%
%
%
\begin{proof}
The estimate is a simple consequence of the error bounds
\eqref{Lasso1} and \eqref{Dinner0}.
\end{proof}
%
%
%
%
%
%
\begin{remark}
For an optimal, logarithmic-type choice of the parameter $\delta$ in \eqref{ModelError} and of
the parameters $\epsilon_1$ and $\epsilon_2$ in \eqref{Final_Estimate}, we refer the reader to
the discussion in Remark~3 of \cite{KZ2013a}. 
\end{remark}
%
%
%
\par\medskip
\par\noindent
{\bf Acknowledgements.}
Work supported by the Research Grant no. 3570/
THALES-AMOSICSS: `Analysis, Modeling and Simulations of Complex and
Stochastic Systems' to the University of Crete co-financed by the European Union
(European Social Fund-ESF) and Greek National Funds through the Operational
Program ``Education and Lifelong Learning" of the National Strategic Reference
Framework (NSRF)-Research Funding Program: THALES-Investing in knowledge
society throught the European Social Fund.
%
%
%
%
%
\def\cprime{$'$}
\def\cprime{$'$}
\end{document}